\font \sevenrm=cmr7
\font \fiverm=cmr5

\documentclass[11pt, english]{smfart}
\usepackage{amsfonts}
\usepackage{amssymb}
\usepackage{amsmath}
\usepackage{wasysym}
\usepackage[all]{xy}               %xypic macro for latex2.09
\usepackage{color}
\usepackage{colordvi}
\usepackage{graphicx}
\usepackage{xspace}
\usepackage{axodraw}

 \newcommand{\nc}{\newcommand}

 {\everymath{\displaystyle\everymath{}}\array}%
 {\endarray}

\newtheorem{thm}{Theorem}

\newtheorem{cor}[thm]{Corollary}

\newtheorem{lem}[thm]{Lemma}
\newtheorem{prop}[thm]{Proposition}
\newtheorem{defn}{Definition}
\newtheorem{rmk}[thm]{Remark}

\nc{\comment}[1]{[[{\tt #1}]] }
\nc{\Cal}[1]{{\mathcal {#1}}}
\nc{\mop}[1]{\mathop{\hbox {\rm #1} }\nolimits}
\nc{\gmop}[1]{\mathop{\hbox {\bf #1} }\nolimits}

\nc{\smop}[1]{\mathop{\hbox {\sevenrm #1} }\nolimits}
\nc{\ssmop}[1]{\mathop{\hbox {\fiverm #1} }\nolimits}
\nc{\mopl}[1]{\mathop{\hbox {\rm #1} }\limits}
\nc{\smopl}[1]{\mathop{\hbox {\sevenrm #1} }\limits}
\nc{\ssmopl}[1]{\mathop{\hbox {\fiverm #1} }\limits}
\nc{\frakg}{{\frak g}}
\nc{\g}[1]{{\frak {#1}}}
\def \restr#1{\mathstrut_{\textstyle |}\raise-6pt\hbox{$\scriptstyle #1$}}
\def \srestr#1{\mathstrut_{\scriptstyle |}\hbox to
  -1.5pt{}\raise-4pt\hbox{$\scriptscriptstyle #1$}}
\nc{\wt}{\widetilde} \nc{\wh}{\widehat}

\nc{\redtext}[1]{\textcolor{red}{\tt[[#1]]}}
\nc{\bluetext}[1]{\textcolor{blue}{#1}}

\nc\fleche[1]{\mathop{\hbox to #1 mm{\rightarrowfill}}\limits}
\nc{\ignore}[1]{}
\def\semi{\mathrel{\times}\kern -.85pt\joinrel\mathrel{\raise
    1.4pt\hbox{${\scriptscriptstyle |}$}}}
\nc\R{{\mathbb R}}
\nc\N{{\mathbb N}}
\nc\T{{\mathbb T}}
\nc\CC{{\mathbb C}}
\nc\inver{^{-1}}
\nc\point{\hbox{\bf .}}
\nc\un{\hbox{\bf 1}}
\def\diagramme #1{\vskip 4mm \centerline {#1} \vskip 4mm}
\def\fleche#1{\mathop{\hbox to #1 mm{\rightarrowfill}}\limits}
\def\gfleche#1{\mathop{\hbox to #1 mm{\leftarrowfill}}\limits}
\def\inj#1{\mathop{\hbox to #1 mm{$\lhook\joinrel$\rightarrowfill}}\limits}
\def\ginj#1{\mathop{\hbox to #1 mm{\leftarrowfill$\joinrel\rhook$}}\limits}
\def\surj#1{\mathop{\hbox to #1 mm{\rightarrowfill\hskip 2pt\llap{$\rightarrow$}}}\limits}
\def\gsurj#1{\mathop{\hbox to #1 mm{\rlap{$\leftarrow$}\hskip 2pt \leftarrowfill}}\limits}
\def\racine{{\scalebox{0.3}{ %%%%%%%%%%%%%%%%%%%%%%%%%%%%%%%%%\ta1
\begin{picture}(12,12)(38,-38)
\SetWidth{0.5} \SetColor{Black} \Vertex(45,-33){5.66}
\end{picture}}}}
  
 \def\arbrea{\,{\scalebox{0.15}{ 
  \begin{picture}(8,55) (370,-248)
    \SetWidth{2}
    \SetColor{Black}
    \Line(374,-244)(374,-200)
    \Vertex(374,-197){9}
    \Vertex(375,-245){12}
  \end{picture}
}}\,}

 \def\arbreba{\,{\scalebox{0.15}{ 
\begin{picture}(8,106) (370,-197)
    \SetWidth{2}
    \SetColor{Black}
    \Line(374,-193)(374,-149)
    \Vertex(374,-146){9}
    \Vertex(375,-194){12}
    \Line(374,-142)(374,-98)
    \Vertex(374,-95){9}
  \end{picture}
}}\,}

 \def\arbrebb{\,{\scalebox{0.15}{ 
  \begin{picture}(48,48) (349,-255)
    \SetWidth{2}
    \SetColor{Black}
    \Vertex(375,-252){12}
    \Line(376,-250)(395,-215)
    \Line(373,-251)(354,-214)
    \Vertex(353,-211){9}
    \Vertex(395,-213){9}
  \end{picture}
}}}

\def\arbreca{\,{\scalebox{0.15}{
\begin{picture}(8,156) (370,-147)
    \SetWidth{2}
    \SetColor{Black}
    \Line(374,-143)(374,-99)
    \Vertex(374,-96){9}
    \Vertex(375,-144){12}
    \Line(374,-92)(374,-48)
    \Vertex(374,-45){9}
    \Line(374,-42)(374,2)
    \Vertex(374,5){9}
  \end{picture}
}}\,}

\def\arbrecb{\,{\scalebox{0.15}{
\begin{picture}(48,94) (349,-255)
\SetWidth{2}
\SetColor{Black}
\Line(376,-204)(395,-169)
\Line(373,-205)(354,-168)
\Vertex(353,-165){9}
\Vertex(395,-167){9}
\Vertex(374,-205){9}
\Line(374,-246)(374,-209)
\Vertex(374,-252){12}
\end{picture}}}\,}

\def\arbrecc{\,{\scalebox{0.15}{
 \begin{picture}(48,98) (349,-205)
    \SetWidth{2}
    \SetColor{Black}
    \Vertex(375,-202){12}
    \Line(376,-200)(395,-165)
    \Line(373,-201)(354,-164)
    \Vertex(353,-161){9}
    \Vertex(395,-163){9}
    \Line(353,-160)(353,-113)
    \Vertex(353,-111){9}
  \end{picture}
}}\,}

\def\arbrecd{\,{\scalebox{0.15}{
\begin{picture}(48,52) (349,-251)
    \SetWidth{2}
    \SetColor{Black}
    \Vertex(375,-248){12}
    \Line(376,-246)(395,-211)
    \Line(373,-247)(354,-210)
    \Vertex(353,-207){9}
    \Vertex(395,-209){9}
    \Line(375,-247)(375,-206)
    \Vertex(376,-203){9}
  \end{picture}
 }}\,}

\def\arbreda{\,{\scalebox{0.15}{
\begin{picture}(8,204) (370,-99)
    \SetWidth{2}
    \SetColor{Black}
    \Line(374,-95)(374,-51)
    \Vertex(374,-48){9}
    \Vertex(375,-96){12}
    \Line(374,-44)(374,0)
    \Vertex(374,3){9}
    \Line(374,6)(374,50)
    \Vertex(374,53){9}
    \Line(374,53)(374,98)
    \Vertex(374,101){9}
  \end{picture}
}}\,}

\def\arbredb{\,{\scalebox{0.15}{
\begin{picture}(48,135) (349,-255)
    \SetWidth{2}
    \SetColor{Black}
    \Line(376,-163)(395,-128)
    \Line(373,-164)(354,-127)
    \Vertex(353,-124){9}
    \Vertex(395,-126){9}
    \Vertex(374,-164){9}
    \Line(374,-205)(374,-168)
    \Vertex(374,-207){9}
    \Line(374,-248)(374,-211)
    \Vertex(374,-252){12}
  \end{picture}
}}\,}

\def\arbredc{\,{\scalebox{0.15}{
 \begin{picture}(48,150) (349,-205)
    \SetWidth{2}
    \SetColor{Black}
    \Line(376,-148)(395,-113)
    \Line(373,-149)(354,-112)
    \Vertex(353,-109){9}
    \Vertex(395,-111){9}
    \Line(353,-108)(353,-61)
    \Vertex(353,-59){9}
    \Line(374,-200)(374,-153)
    \Vertex(374,-149){9}
    \Vertex(374,-202){12}
  \end{picture}
}}\,}

\def\arbredd{\,{\scalebox{0.15}{
 \begin{picture}(48,99) (349,-251)
    \SetWidth{2}
    \SetColor{Black}
    \Line(376,-199)(395,-164)
    \Line(373,-200)(354,-163)
    \Vertex(353,-160){9}
    \Vertex(395,-162){9}
    \Vertex(376,-156){9}
    \Vertex(376,-248){12}
    \Line(375,-245)(375,-204)
    \Line(375,-200)(375,-159)
    \Vertex(375,-201){9}
  \end{picture}
}}\,}

\def\arbrede{\,{\scalebox{0.15}{
 \begin{picture}(48,153) (349,-150)
    \SetWidth{2}
    \SetColor{Black}
    \Vertex(375,-147){12}
    \Line(376,-145)(395,-110)
    \Line(373,-146)(354,-109)
    \Vertex(353,-106){9}
    \Vertex(395,-108){9}
    \Line(353,-105)(353,-58)
    \Vertex(353,-56){9}
    \Line(353,-52)(353,-5)
    \Vertex(353,-1){9}
  \end{picture}
}}\,}

\def\arbredf{\,{\scalebox{0.15}{
\begin{picture}(48,98) (349,-205)
    \SetWidth{2}
    \SetColor{Black}
    \Vertex(375,-202){12}
    \Line(376,-200)(395,-165)
    \Line(373,-201)(354,-164)
    \Vertex(353,-161){9}
    \Vertex(395,-163){9}
    \Line(353,-160)(353,-113)
    \Vertex(353,-111){9}
    \Line(395,-159)(395,-112)
    \Vertex(395,-111){9}
  \end{picture}
}}\,}

\def\arbredz{\,{\scalebox{0.15}{
  \begin{picture}(68,88) (329,-215)
    \SetWidth{2}
    \SetColor{Black}
    \Vertex(375,-212){12}
    \Line(376,-210)(395,-175)
    \Line(373,-211)(354,-174)
    \Vertex(353,-171){9}
    \Vertex(395,-173){9}
    \Line(351,-168)(332,-131)
    \Line(355,-168)(374,-133)
    \Vertex(333,-131){9}
    \Vertex(374,-131){9}
  \end{picture}
}}\,}

\def\arbredg{\,{\scalebox{0.15}{
\begin{picture}(48,98) (349,-205)
    \SetWidth{2}
    \SetColor{Black}
    \Vertex(375,-202){12}
    \Line(376,-200)(395,-165)
    \Line(373,-201)(354,-164)
    \Vertex(353,-161){9}
    \Vertex(395,-163){9}
    \Line(375,-201)(375,-160)
    \Vertex(376,-157){9}
    \Vertex(376,-111){9}
    \Line(375,-155)(375,-114)
  \end{picture}
}}\,}

\def\arbredh{\,{\scalebox{0.15}{
 \begin{picture}(90,46) (330,-257)
    \SetWidth{2}
    \SetColor{Black}
    \Vertex(375,-254){12}
    \Line(376,-252)(395,-217)
    \Vertex(395,-215){9}
    \Line(374,-254)(335,-226)
    \Vertex(334,-224){9}
    \Line(375,-252)(356,-215)
    \Vertex(355,-215){9}
    \Line(374,-255)(417,-227)
    \Vertex(418,-225){9}
  \end{picture}
}}\,}

\def\operade{\,{\scalebox{0.6}{  
  \begin{picture}(148,136) (9,-15)
    \SetWidth{1}
    \SetColor{Black}
    \GBox(35,71)(103,89){0.882}
    \GBox(9,27)(138,46){0.882}
    \Line(60,120)(60,89)
    \Line(44,120)(44,89)
    \Line(91,120)(91,89)
    \Line(77,120)(77,89)
    \Line(16,46)(16,120)
    \Line(113,46)(113,120)
    \Line(127,46)(127,120)
    \Line(67,70)(67,46)
    \Line(73,27)(73,2)
    \Text(63,75)[lb]{\Large{\Black{$b$}}}
    \Text(68,31)[lb]{\Large{\Black{$a$}}}
    \Text(-5,-15)[lb]{\Large{\Black{Partial composition $a\circ_i b$}}}
    \Text(18,49)[lb]{\large{\Black{$1$}}}
    \Text(20,61)[lb]{\large{\Black{$\cdots$}}}
    \Text(96,61)[lb]{\large{\Black{$\cdots$}}}
    \Text(71,49)[lb]{\large{\Black{$i$}}}
    \Text(130,49)[lb]{\large{\Black{$k$}}}
  \end{picture}
}}\,}

\def\asoperade{\,{\scalebox{0.6}{  
\begin{picture}(585,186) (104,-20)
    \SetWidth{1}
    \SetColor{Black}
    \GBox(165,109)(227,126){0.882}
    \GBox(134,65)(253,83){0.882}
    \GBox(104,17)(319,36){0.882}
    \Line(173,156)(173,126)
    \Line(182,156)(182,126)
    \Line(192,156)(192,126)
    \Line(219,156)(219,126)
    \Line(138,156)(138,84)
    \Line(237,156)(237,84)
    \Line(248,156)(248,84)
    \Line(108,155)(108,36)
    \Line(262,154)(262,35)
    \Line(305,156)(305,37)
    \Line(194,109)(194,83)
    \Line(194,64)(194,36)
    \Line(206,17)(206,-3)
    \GBox(474,18)(689,37){0.882}
    \GBox(165,109)(227,126){0.882}
    \GBox(165,109)(227,126){0.882}
    \GBox(492,64)(554,81){0.882}
    \GBox(613,64)(675,81){0.882}
    \Line(502,112)(502,82)
    \Line(516,112)(516,82)
    \Line(544,112)(544,82)
    \Line(630,111)(630,81)
    \Line(641,111)(641,81)
    \Line(667,112)(667,82)
    \Line(621,111)(621,81)
    \Line(522,63)(522,37)
    \Line(644,64)(644,36)
    \Line(581,18)(581,-1)
    \Line(478,112)(478,38)
    \Line(562,111)(562,37)
    \Line(606,111)(606,37)
    \Line(685,112)(685,38)
    \Text(202,24)[lb]{\Large{\Black{$a$}}}
    \Text(190,70)[lb]{\Large{\Black{$b$}}}
    \Text(192,115)[lb]{\Large{\Black{$c$}}}
    \Text(578,24)[lb]{\Large{\Black{$a$}}}
    \Text(520,69)[lb]{\Large{\Black{$b$}}}
    \Text(641,69)[lb]{\Large{\Black{$c$}}}
    \Text(145,-20)[lb]{\Large{\Black{Nested associativity}}}
    \Text(525,-16)[lb]{\Large{\Black{Disjoint associativity}}}
    \Text(112,43)[lb]{\large{\Black{$1$}}}
    \Text(197,44)[lb]{\large{\Black{$i$}}}
    \Text(308,44)[lb]{\large{\Black{$k$}}}
    \Text(111,67)[lb]{\large{\Black{$\cdots$}}}
    \Text(141,91)[lb]{\large{\Black{$1$}}}
    \Text(197,90)[lb]{\large{\Black{$j$}}}
    \Text(251,90)[lb]{\large{\Black{$l$}}}
    \Text(275,67)[lb]{\large{\Black{$\cdots$}}}
    \Text(482,44)[lb]{\large{\Black{$1$}}}
    \Text(526,45)[lb]{\large{\Black{$i$}}}
    \Text(650,44)[lb]{\large{\Black{$j$}}}
    \Text(505,88)[lb]{\large{\Black{$1$}}}
    \Text(547,88)[lb]{\large{\Black{$l$}}}
    \Text(519,104)[lb]{\large{\Black{$\cdots$}}}
    \Text(644,104)[lb]{\large{\Black{$\cdots$}}}
    \Text(194,150)[lb]{\large{\Black{$\cdots$}}}
    \Text(575,76)[lb]{\large{\Black{$\cdots$}}}
  \end{picture}
}}\,}

\def\gcoperade{\,{\scalebox{0.6}{
\begin{picture}(500,181) (102,-149)
    \SetWidth{1}
    \SetColor{Black}
    \GBox(102,-100)(602,-72){0.882}
    \GBox(132,-29)(198,-10){0.882}
    \GBox(230,-29)(313,-10){0.882}
    \GBox(395,-29)(550,-10){0.882}
    \Line(141,31)(141,-10)
    \Line(157,31)(157,-10)
    \Line(181,31)(181,-10)
    \Line(241,31)(241,-10)
    \Line(259,31)(259,-10)
    \Line(279,31)(279,-10)
    \Line(300,31)(300,-10)
    \Line(405,31)(405,-10)
    \Line(421,30)(421,-10)
    \Line(448,29)(448,-10)
    \Line(486,31)(486,-10)
    \Line(534,31)(534,-10)
    \Text(338,-91)[lb]{\Large{\Black{$a$}}}
    \Text(159,-24)[lb]{\Large{\Black{$b_1$}}}
    \Text(266,-24)[lb]{\Large{\Black{$b_2$}}}
    \Text(463,-24)[lb]{\Large{\Black{$b_n$}}}
    \Line(166,-29)(166,-72)
    \Line(270,-29)(270,-72)
    \Line(471,-29)(471,-72)
    \Text(335,-20)[lb]{\Large{\Black{$\cdots$}}}
    \Line(344,-100)(344,-123)
    \Text(209,-149)[lb]{\Large{\Black{Global composition $\gamma(a;b_1,b_2,\ldots,b_n)$}}}
  \end{picture}
}}\,}

\begin{document}
%%%%%%%%%%%%%%%%%%%%%%%%%%%%%%%%%%%%%%%%%%%%%
%%%%%%%%%%%%%%%%%%%%%%%%%%%%%%%%%%%%%%%%%%%%%

\title[Algebraic background]{Algebraic background for numerical methods, control theory and renormalization}
         
\author{Dominique Manchon}
\address{Universit\'e Blaise Pascal,
         C.N.R.S.-UMR 6620, BP 80026,
         63171 Aubi\`ere, France}       
         \email{manchon@math.univ-bpclermont.fr}
         \urladdr{http://math.univ-bpclermont.fr/~manchon/}

%%%%%%%%%%%%%%%%%%%%%%%%%%%%%%%%%%%%%%%%%%%%%%%%%%%%%%%%%%%%%%%%%%%
\date{August 10th  2012}
%%%%%%%%%%%%%%%%%%%%%%%%%%%%%%%%%%%%%%%%%%%%%%%%%%%%%%%%%%%%%%%%%%%
\begin{abstract}
We review some important algebraic structures which appear in a priori remote areas of
Mathematics, such as control theory, numerical methods for solving
differential equations, and renormalization in Quantum Field Theory. Starting
with connected Hopf algebras we will also introduce augmented operads, and
devote a substantial part of this chapter to pre-Lie algebras. Other related
algebraic structures (Rota-Baxter and dendriform algebras, NAP algebras) will
be also mentioned.
\end{abstract}

\maketitle

%%%%%%%%%%%%%%%%%%%%%%%%%%%%%%%%%%%%%%%%%%%%%%%%%%%%%%%%%%%%%%%%%%%

\tableofcontents

%%%%%%%%%%%%%%%%%%%%%%%%%%%%%%%%%%%%%%%%%%%%%%%%%%%%%%%%%%%%%%%%%%%

\section{Introduction}
\label{sect:intro}
It is known since the pioneering work of A. Cayley in the Nineteenth century
\cite{Cay} that rooted trees and vector fields on the affine space are closely
related. Surprisingly enough, rooted trees also revealed to be a fundamental
tool for studying not only the integral curves of vector fields, but also their
Runge-Kutta numerical approximations \cite{B63}.\\

The rich algebraic structure
of the $k$-vector space $\Cal T$ spanned by rooted trees (where $k$ is some field
of characteristic zero) can be, in a nutshell, described as follows: $\Cal T$ is
both the free pre-Lie algebra with one generator and the free Non-Associative
Permutative algebra with one generator \cite{ChaLiv}, \cite{DL}, and moreover
there are two other pre-Lie structures on $\Cal T$, of operadic nature, which
show strong compatibility with the first pre-Lie (resp. the NAP)
structure (\cite{ChaLiv}, \cite{CEM}, \cite{MS}). The Hopf algebra of
coordinates on the Butcher group of \cite{B63}, i.e. the graded dual of the enveloping algebra $\Cal U(\Cal T)$ (with
respect to the Lie bracket given by the the first pre-Lie structure) was first investigated in
\cite {D}, and intensively studied by D. Kreimer for renormalization purposes
in Quantum Field Theory (\cite{CK1}, \cite{K}, see also \cite{Br}).\\

This chapter is organized as follows: the first section is devoted to general
connected graded or filtered Hopf algebras, including renormalization of their characters. The
second section gives a short presentation of operads in the symmetric monoidal
category of vector spaces, and the third section will treat pre-Lie algebras
in some detail: in particular we will give a ``pedestrian'' proof of the
Chapoton-Livernet theorem on free pre-Lie algebras. In
the last section Rota-Baxter, dendriform and NAP algebras will be introduced.\\

\noindent\textbf{Acknowledgements}: We thank the referees for their careful reading and their suggestions which greatly helped to improve the final version. We also thank Nicol\'as Andruskiewitsch for illuminating discussions on Hopf algebra filtrations.
%%%%%
\section{Hopf algebras: general properties}
%%%%%
We choose a base field $k$ of characteristic zero. Most of the material here
is borrowed from \cite{Man}, to which we refer for more details.
\subsection{Algebras}
A $k$-algebra is by definition a $k$-vector space $A$ together with a bilinear map $m:A\otimes A\to A$ which is {\sl associative\/}. The associativity is expressed by the commutativity of the following diagram:
\diagramme{
\xymatrix{
A\otimes A\otimes A \ar[d]^{I\otimes m}\ar[r]^{m\otimes I}
	&A\otimes A \ar[d]^m	\\
A\otimes A \ar[r]^m	& A}
}
The algebra $A$ is {\sl unital\/} if moreover there is a unit $\un$ in it. This is expressed by the commutativity of the following diagram:
\diagramme{
\xymatrix{
k\otimes A \ar[r]^{u\otimes I} \ar[dr]^\sim	& A\otimes A \
 \ar[d]_m	&A\otimes k \ar[l]_{I\otimes u}\ar[dl]_\sim	\\
&A&}
}
where $u$ is the map from $k$ to $A$ defined by $u(\lambda)=\lambda\un$. The algebra $A$ is {\sl commutative\/} if $m\circ \tau=m$, where $\tau:A\otimes A\to A\otimes A$ is the {\sl flip\/}, defined by $\tau(a\otimes b)=b\otimes a$.\\

A subspace $J\subset A$ is called a {\sl subalgebra\/} (resp. a {\sl left ideal, right ideal, two-sided ideal\/}) of $A$ if $m(J\otimes J)$ (resp. $m(A\otimes J)$, $m(J\otimes A)$, $m(J\otimes A+A\otimes J$) is included in $J$.\\

To any vector space $V$ we can associate its {\sl tensor algebra\/} $T(V)$. As a vector space it is defined by:
$$T(V)=\bigoplus_{k\ge 0}V^{\otimes k},$$
with $V^{\otimes 0}=k$ and $V^{\otimes k+1}:=V\otimes V^{\otimes k}.$ The product is given by the {\sl concatenation\/}:
$$m(v_1\otimes\cdots\otimes v_p,\, v_{p+1}\otimes\cdots\otimes v_{p+q})=
v_1\otimes\cdots\otimes v_{p+q}.$$
The embedding of $k=V^{\otimes 0}$ into $T(V)$ gives the unit map $u$. Tensor
algebra $T(V)$ is also called the {\sl free (unital) algebra generated by
  $V$\/}. This algebra is characterized by the following universal property:
for any linear map $\varphi$ from $V$ to a unital algebra $A$ there is a
unique unital algebra morphism $\wt\varphi$ from $T(V)$ to $A$ extending $\varphi$.
\\

Let $A$ and $B$ be unital $k$-algebras. We put a unital algebra structure on $A\otimes B$ in the following way:
$$(a_1\otimes b_1).(a_2\otimes b_2)=a_1a_2\otimes b_1b_2.$$
The unit element $\un_{A\otimes B}$ is given by $\un_A\otimes \un_B$, and the
associativity is clear. This multiplication is thus given by:
$$m_{A\otimes B}=(m_A\otimes m_B)\circ \tau_{23},$$
where $\tau_{23}: A\otimes B\otimes A\otimes B\to A\otimes A\otimes B\otimes B$ is defined by the flip of the two middle factors:
$$\tau_{23}(a_1\otimes b_1\otimes a_2\otimes b_2)=a_1\otimes a_2\otimes b_1\otimes b_2.$$
\subsection{Coalgebras}
Coalgebras are objects wich are somehow dual to algebras: axioms for
coalgebras are derived from axioms for algebras by reversing the arrows of the
corresponding diagrams:\\

A $k$-coalgebra is by definition a $k$-vector space $C$ together with a bilinear map $\Delta:C\to C\otimes C$ which is {\sl co-associative\/}. The co-associativity is expressed by the commutativity of the following diagram:
\diagramme{
\xymatrix{
C\otimes C\otimes C 
	&C\otimes C \ar[l]_{\Delta\otimes I}	\\
C\otimes C \ar[u]_{I\otimes \Delta}	& C\ar[l]_{\Delta}\ar[u]_{\Delta}}
}
Coalgebra $C$ is {\sl co-unital\/} if moreover there is a co-unit $\varepsilon:C\to k$ such that the following diagram commutes:
\diagramme{
\xymatrix{
k\otimes C  	& C\otimes C \ar[l]_{\varepsilon\otimes I}\
 	\ar[r]^{I\otimes \varepsilon} &C\otimes k 	\\
&C\ar[u]^\Delta \ar[ul]_\sim \ar[ur]^\sim &}
}
A subspace $J\subset C$ is called a {\sl subcoalgebra\/} (resp. a {\sl left
coideal, right coideal, two-sided coideal\/}) of $C$ if $\Delta(J)$ is
contained in $J\otimes J$ (resp. $C\otimes J$,  $J\otimes C$, $J\otimes C+C\otimes J$) is included in $J$. The duality alluded to above can be made more precise:
\begin{prop}\label{dual1}
\begin{enumerate}
\item The linear dual $C^*$ of a co-unital coalgebra $C$ is a unital algebra, with product (resp. unit map) the transpose of the coproduct (resp. of the co-unit). 
\item Let $J$ be a linear subspace of $C$. Denote by $J^\perp$ the orthogonal of $J$ in $C^*$. Then:
\begin{itemize}
\item
$J$ is a two-sided coideal if and only if $J^\perp$ is a subalgebra of $C^*$.
\item
$J$ is a left coideal if and only if $J^\perp$ is a left ideal of $C^*$.
\item
$J$ is a right coideal if and only if $J^\perp$ is a right ideal of $C^*$.
\item
$J$ is a subcoalgebra if and only if $J^\perp$ is a two-sided ideal of $C^*$.
\end{itemize}
\end{enumerate}
\end{prop}
\begin{proof}
For any subspace $K$ of $C^*$ we shall denote by $K^\perp$ the subspace of those elements of $C$ on which any element of $K$ vanishes. It coincides with the intersection of the orthogonal of $K$ with $C$, via the canonical embedding $C\inj 6 C^{**}$. So we have for any linear subspaces $J\subset C$ and $K\subset C^*$:
$$J^{\perp\perp}=J,\hskip 20mm K^{\perp\perp}\supset K.$$  
Suppose that $J$ is a two-sided coideal. Take any $\xi,\eta$ in $J^\perp$. For any $x\in J$ we have:
$$<\xi\eta,x>=<\xi\otimes\eta,\Delta x>=0,$$
as $\Delta x\subset J\otimes C+C\otimes J$. So $J^\perp$ is a subalgebra of $C^*$. Conversely if $J^\perp$ is a subalgebra then:
$$\Delta J\subset (J^\perp\otimes J^\perp)^\perp=J\otimes C+C\otimes J,$$
which proves the first assertion. We leave it to the reader as an exercice to
prove the three other assertions along the same lines. 
\end{proof}
\noindent
Dually we have the following:
\begin{prop}\label{dual2}
Let $K$ a linear subspace of $C^*$. Then:\begin{itemize}
\smallskip
\item
$K^\perp$ is a two-sided coideal if and only if $K$ is a subalgebra of $C^*$.
\item
$K^\perp$ is a left coideal if and only if $K$ is a left ideal of $C^*$.
\item
$K^\perp$ is a right coideal if and only if $K$ is a right ideal of $C^*$.
\item
$K^\perp$ is a subcoalgebra if and only if $K$ is a two-sided ideal of $C^*$.
\end{itemize}
\end{prop}
\begin{proof}
The linear dual $(C\otimes C)^*$ naturally contains the tensor product $C^*\otimes C^*$. Take as a multiplication the restriction of $^t\!\Delta$ to $C^*\otimes C^*$:
$$m=^t\!\Delta:C^*\otimes C^*\longrightarrow C^*,$$
and put $u=^t\!\varepsilon:k\to C^*$. It is easily seen, by just reverting the
arrows of the corresponding diagrams, that coassociativity of $\Delta$ implies
associativity of $m$, and that the co-unit property for $\varepsilon$ implies
that $u$ is a unit. 
\end{proof}
Note that the duality
property is not perfect: if the linear dual of a coalgebra is always an
algebra, the linear dual of an algebra is not in general a coalgebra. However
the {\sl restricted dual\/} $A^\circ$ of an algebra $A$ is a coalgebra. It is
defined as the space of linear forms on $A$ vanishing on some
finite-codimensional ideal \cite{Sw69}.\\

The coalgebra $C$ is {\sl cocommutative\/} if $\tau\circ\Delta=\Delta$, where $\tau:C\otimes C\to C\otimes C$ is the flip. It will be convenient to use {\sl Sweedler's notation\/}:
$$\Delta x=\sum_{(x)}x_1\otimes x_2.$$
Cocommutativity expresses then as:
$$\sum_{(x)}x_1\otimes x_2=\sum_{(x)}x_2\otimes x_1.$$
Coassociativity reads in Sweedler's notation:
$$(\Delta\otimes I)\circ\Delta(x)=\sum_{(x)}x_{1:1}\otimes x_{1:2}\otimes x_2=
	\sum_{(x)}x_1\otimes x_{2:1}\otimes x_{2:2}
	=(I\otimes\Delta)\circ\Delta(x),$$
We shall sometimes write the iterated coproduct as:
$$\sum_{(x)}x_1\otimes x_2\otimes x_3.$$
Sometimes we shall even mix the two ways of using Sweedler's notation for the
iterated coproduct, in the case we want to keep partially track of how we have
constructed it \cite{DNR}. For example,
\begin{eqnarray*}
\Delta_3(x)	&=&(I\otimes\Delta\otimes I)\circ(\Delta\otimes I)
	\circ\Delta(x)	\\
			&=&(I\otimes\Delta\otimes I)(\sum_{(x)}
		x_1\otimes x_2\otimes x_3)	\\
			&=&\sum_{(x)}
		x_1\otimes x_{2:1}\otimes x_{2:2}\otimes x_3.
\end{eqnarray*}
To any vector space $V$ we can associate its {\sl tensor coalgebra\/} $T^c(V)$. It is isomorphic to $T(V)$ as a vector space. The coproduct is given by the {\sl deconcatenation\/}:
$$\Delta(v_1\otimes\cdots\otimes v_n)=\sum_{p=0}^n
(v_1\otimes\cdots\otimes v_p) \bigotimes (v_{p+1}\otimes\cdots\otimes v_n).$$
The co-unit is given by the natural projection of $T^c(V)$ onto $k$.
\\ \\
Let $C$ and $D$ be unital $k$-coalgebras. We put a co-unital coalgebra structure on $C\otimes D$ in the following way: the comultiplication is given by:
$$\Delta_{C\otimes D}=\tau_{23}\circ(\Delta_C\otimes \Delta_D) ,$$
where $\tau_{23}$ is again the flip of the two middle factors, and the
co-unity is given by $\varepsilon_{C\otimes
  D}=\varepsilon_C\otimes\varepsilon_D$.
\subsection{Convolution product}
Let $A$ be an algebra and $C$ be a coalgebra (over the same field $k$). Then
there is an associative product on the space $\Cal L(C,A)$ of linear maps from
$C$ to $A$ called the {\sl convolution product\/}. It is given by:
$$\varphi*\psi=m_{A}\circ(\varphi\otimes\psi)\circ\Delta_{C}.$$
In Sweedler's notation it reads:
$$\varphi*\psi(x)=\sum_{(x)}\varphi(x_1)\psi(x_2).$$
The associativity is a direct consequence of both associativity of $A$ and
coassociativity of $C$.
\subsection{Bialgebras and Hopf algebras}
A (unital and co-unital) {\sl bialgebra\/} is a vector space $\Cal H$ endowed with a structure of unital algebra $(m,u)$ and a structure of co-unital coalgebra $(\Delta,\varepsilon)$ which are compatible. The compatibility requirement is that $\Delta$ is an algebra morphism (or equivalently that $m$ is a coalgebra morphism), $\varepsilon$ is an algebra morphism and $u$ is a coalgebra morphism. It is expressed by the commutativity of the three following diagrams:
\diagramme{
\xymatrix{\Cal H\otimes\Cal H\otimes\Cal H\otimes\Cal H
\ar[rr]^{\tau_{23}}	&&\Cal H\otimes\Cal H\otimes\Cal H\otimes\Cal H
\ar[d]^{m\otimes m}	\\
\Cal H\otimes \Cal H	\ar[u]_{\Delta\otimes\Delta} \ar[r]_m
&\Cal H \ar[r]_{\Delta}	&\Cal H\otimes\Cal H}
}
\diagramme{
\xymatrix{\Cal H\otimes\Cal H \ar[d]^m \ar[r]^{\varepsilon\otimes\varepsilon}
	&k\otimes k	\ar[d]^\sim	&&&\Cal H\otimes \Cal H	&k\otimes k \ar[l]_{u\otimes u}\\
\Cal H\ar[r]^\varepsilon &k	&&&\Cal H \ar[u]_\Delta	&k\ar[l]_u \ar[u]_\sim}
}
A {\sl Hopf algebra\/} is a bialgebra $\Cal H$ together with a linear map $S:\Cal H\to \Cal H$ called the {\sl antipode\/}, such that the following diagram commutes:
\diagramme{
\xymatrix{&\Cal H\otimes\Cal H	\ar[rr]^{S\otimes I}
				&&\Cal H\otimes \Cal H\ar[dr]^{m}	& \\
\Cal H\ar[rr]^\varepsilon \ar[dr]^\Delta \ar[ur]^\Delta
				&&	k\ar[rr]^u	&&\Cal H\\
&\Cal H\otimes\Cal H	\ar[rr]^{I\otimes S}
				&&\Cal H\otimes \Cal H\ar[ur]^{m}	&}
}
In Sweedler's notation it reads:
$$\sum_{(x)}S(x_1)x_2=\sum_{(x)}x_1S(x_2)=(u\circ\varepsilon)(x).$$
In other words the antipode is an inverse of the identity $I$ for the
convolution product on $\Cal L(\Cal H,\Cal H)$. The unit for the convolution is the map $u\circ\varepsilon$.
\\

A {\sl primitive element\/} in a bialgebra $\Cal H$ is an element $x$ such
that $\Delta x=x\otimes 1+1\otimes x$. A {\sl grouplike element\/} is a
nonzero element $x$ such that $\Delta x=x\otimes x$. Note that grouplike
elements make sense in any coalgebra.
\\

A {\sl bi-ideal\/} in a bialgebra $\Cal H$ is a two-sided ideal which is also
a two-sided co-ideal. A {\sl Hopf ideal\/} in a Hopf algebra $\Cal H$ is a
bi-ideal $J$ such that $S(J)\subset J$.
\subsection{Some simple examples of Hopf algebras}\label{examples}
\subsubsection{The Hopf algebra of a group}\label{ex1}
Let $G$ be a group, and let $kG$ be the group algebra (over the field $k$). It
is by definition the vector space freely generated by the elements of $G$:
the product of $G$ extends uniquely to a bilinear map from $kG\times kG$ into
$kG$, hence a multiplication $m:kG\otimes kG\to kG$, which is associative. The
neutral element of $G$ gives the unit for $m$. The space $kG$ is also endowed with a co-unital coalgebra structure, given by:
$$\Delta (\sum \lambda_ig_i)=\sum \lambda_i.g_i\otimes g_i$$
and:
$$\varepsilon (\sum \lambda_ig_i)=\sum \lambda_i.$$
This defines the {\sl coalgebra of the set $G$\/}: it does not take into account the extra group structure on $G$, as the algebra structure does.
\begin{prop}  
The vector space $kG$ endowed with the algebra and coalgebra structures defined above is a Hopf algebra. The antipode is given by:
$$S(g)=g\inver, g\in G.$$
\end{prop}
\begin{proof}
The compatibility of the product and the coproduct is an immediate consequence of the following computation: for any $g,h\in G$ we have:
$$\Delta(gh)=gh\otimes gh=(g\otimes g)(h\otimes h)=\Delta g.\Delta h.$$
Now $m(S\otimes I)\Delta(g)=g\inver g=e$ and similarly for $m(I\otimes S)\Delta(g)$. But $e=u\circ\varepsilon(g)$ for any $g\in G$, so the map $S$ is indeed the antipode.
\end{proof}
\begin{rmk}if $G$ were only a semigroup, the same construction would lead to a bialgebra structure on $kG$: the Hopf algebra structure (i.e. the existence of an antipode) reflects the group structure (the existence of the inverse). We have $S^2=I$ in this case, but involutivity of the antipode is not true for general Hopf algebras.
\end{rmk}
\subsubsection{Tensor algebras}\label{ex2}
There is a natural structure of cocommutative Hopf algebra on the tensor algebra $T(V)$ of any vector space $V$. Namely we define the coproduct $\Delta$ as the unique algebra morphism from $T(V)$ into $T(V)\otimes T(V)$ such that:
$$\Delta(1)=1\otimes 1,\hskip 12mm \Delta(x)=x\otimes 1+1\otimes x, \ x\in V.$$
We define the co-unit as the algebra morphism such that $\varepsilon(1)=1$ and $\varepsilon\restr V=0$ This endows $T(V)$ with a cocommutative bialgebra structure. We claim that the principal anti-automorphism:
$$S(x_1\otimes\cdots\otimes x_n)=(-1)^nx_n\otimes\cdots\otimes x_1$$
verifies the axioms of an antipode, so that $T(V)$ is indeed a Hopf
algebra. For $x\in V$ we have $S(x)=-x$, hence $S*I(x)=I*S(x)=0$. As $V$
generates $T(V)$ as an algebra it is easy to conclude.
\subsubsection{Enveloping algebras}\label{ex3}
Let $\g g$ a Lie algebra. The universal enveloping algebra is the quotient of
the tensor algebra $T(\g g)$ by the ideal $J$ generated by $x\otimes
y-y\otimes x-[x,y], \ x,y\in\g g$.
\begin{lem}
$J$ is a Hopf ideal, i.e. $\Delta(J)\subset\Cal H\otimes J+J\otimes\Cal H$ and
$S(J)\subset J$.
\end{lem}
\begin{proof}
The ideal $J$ is generated by primitive elements (according to proposition
\ref{primitives} below), and any ideal generated by primitive elements is a Hopf ideal (very easy and left to the reader).
\end{proof}
The quotient of a Hopf algebra by a Hopf ideal is a Hopf algebra. Hence the
universal enveloping algebra $\Cal U(\g g)$ is a cocommutative Hopf algebra.
\subsection{Some basic properties of Hopf algebras}
We summarize in the proposition below the main properties of the antipode in a Hopf algebra:
\begin{prop}\label{hopf-gen}{\rm (cf. \cite{Sw69} proposition 4.0.1)}
Let $\Cal H$ be a Hopf algebra with multiplication $m$, comultiplication
$\Delta$, unit $u:1\mapsto \un$, co-unit $\varepsilon$ and antipode
$S$. Then:
\begin{enumerate}
\item $S\circ u=u$ and $\varepsilon\circ S=\varepsilon$.
\smallskip
\item $S$ is an algebra antimorphism and a coalgebra antimorphism, i.e. if $\tau$ denotes the flip we have:
$$m\circ (S\otimes S)\circ\tau=S\circ m,\hskip 12mm \tau\circ(S\otimes S)\circ\Delta=\Delta\circ S.$$
\item If $\Cal H$ is commutative or cocommutative, then $S^2=I$.
\end{enumerate}
\end{prop}
\noindent
For a detailed proof, see Chr. Kassel in \cite K.
\begin{prop}\label{primitives}
\begin{enumerate}
\item If $x$ is a primitive element then $S(x)=-x$.
\smallskip
\item The linear subspace $\mop{Prim}\Cal H$ of primitive elements in $\Cal H$ is a Lie algebra.
\end{enumerate}
\end{prop}
\begin{proof}
If $x$ is primitive, then $(\varepsilon\otimes\varepsilon)\circ\Delta (x)=2\varepsilon(x)$. On the other hand, $(\varepsilon\otimes\varepsilon)\circ\Delta (x)=\varepsilon(x)$, so $\varepsilon(x)=0$. Then:
$$0=(u\circ\varepsilon)(x)=m(S\otimes I)\Delta(x)=S(x)+x. $$
Now let $x$ and $y$ be primitive elements of $\Cal H$. Then we can easily
compute:
\begin{eqnarray*}
\Delta(xy-yx)	&=&(x\otimes \un+\un\otimes x)(y\otimes \un
+\un\otimes y)-(y\otimes \un+\un\otimes y)(x\otimes \un+\un\otimes x)	\\
				&=&(xy-yx)\otimes \un+\un\otimes (xy+yx)
+x\otimes y+y\otimes x-y\otimes x-x\otimes y	\\
				&=&(xy-yx)\otimes \un+\un\otimes (xy-yx).
\end{eqnarray*}
\end{proof}
%%%%%
\section{Connected Hopf algebras}
%%%%%
We introduce the crucial property of connectedness for bialgebras. The main
interest resides in the possibility to implement recursive procedures in
connected bialgebras, the induction taking place with respect to a filtration or a grading. An important example of these
techniques is the recursive construction of the antipode, which then ``comes
for free'', showing that any connected bialgebra is in fact a connected Hopf
algebra.
\subsection{Connected graded bialgebras}
A {\sl graded Hopf algebra\/} on $k$ is a graded $k$-vector space:
$$\Cal H=\bigoplus_{n\ge 0}\Cal H_n$$
endowed with a product $m:\Cal H\otimes \Cal H\to\Cal H$, a coproduct
$\Delta:\Cal H\to\Cal H\otimes \Cal H$, a unit $u:k\to\Cal H$, a co-unit
$\varepsilon:\Cal H\to k$ and an antipode $S:\Cal H\to\Cal H$ fulfilling the
usual axioms of a Hopf algebra, and such that:
\begin{eqnarray}
\Cal H_p.\Cal H_q	&\subset& \Cal H_{p+q}	\\
		\Delta(\Cal H_n)	&\subset& \bigoplus_{p+q=n}\Cal H_p\otimes\Cal H_q.
\end{eqnarray}
If we do not ask for the existence of an antipode $\Cal H$ we get the definition of a {\sl graded bialgebra\/}. In a graded bialgebra $\Cal H$ we shall consider the increasing filtration:
$$\Cal H^n=\bigoplus_{p=0}^n\Cal H_p.$$
It is an easy exercice (left to the reader) to prove that the unit $u$ and the co-unit $\varepsilon$ are degree zero maps, i.e. $\un\in\Cal H_0$ and $\varepsilon(\Cal H_n)=\{0\}$ for $n\ge 1$. One also can show that the antipode $S$, when it exists, is also of degree zero, i.e. $S(\Cal H_n)\subset\Cal H_n$. It can be proved as follows: let $S':\Cal H\to\Cal H$ be defined so that $S'(x)$ is the $n^{\smop{th}}$ homogeneous component os $S(x)$ when $x$ is homogeneous of degree $n$. We can write down the coproduct $\Delta (x)$ with the Sweedler notation:
\begin{equation*}
\Delta(x)=\sum_{(x)}x_1\otimes x_2,
\end{equation*}
where $x_1$ and $x_2$ are homogeneous of degree, say, $k$ and $n-k$. We have then:
\begin{equation}
m\circ(S'\otimes \mop{Id})\circ\Delta(x)=\sum_{(x)}S'(x_1)x_2=n^{\smop{th}} \hbox{ component of }\sum_{(x)}S(x_1)x_2=\varepsilon(x)\un.
\end{equation}
Similarly, $m\circ(\mop{Id}\otimes S')\circ\Delta(x)=\varepsilon(x)\un.$ By uniqueness of the antipode we get then $S'=S$.\\

Suppose moreover that $\Cal H$ is {\sl connected\/}, i.e. $\Cal H_{0}$ is one-dimensional. Then we have:
$$\mop{Ker}\varepsilon=\bigoplus_{n\ge 1}\Cal H_n.$$
\begin{prop}\label{coproduit-connexe}
For any $x\in\Cal H^n, n\ge 1$ we can write:
$$\Delta x=x\otimes\un+\un\otimes x+\wt\Delta x,\hskip 12mm 
\wt\Delta x\in\bigoplus_{p+q=n,\, p\not= 0,\, q\not= 0}\Cal H_p\otimes\Cal H_q.$$
The map $\wt\Delta$ is coassociative on $\mop{Ker}\varepsilon$ and $\wt\Delta_k=(I^{\otimes k-1}\otimes\wt\Delta)(I^{\otimes k-2}\otimes\wt\Delta)...\wt\Delta$ sends $\Cal H^n$ into $(\Cal H^{n-k})^{\otimes k+1}$.
\end{prop}
\begin{proof}
Thanks to connectedness we clearly can write:
$$\Delta x=a(x\otimes 1)+b(1\otimes x)+\wt\Delta x$$
with $a,b\in k$ and $\wt\Delta
x\in\mop{Ker}\varepsilon\otimes\mop{Ker}\varepsilon$. The co-unit property
then tells us that, with $k\otimes \Cal H$ and $\Cal H\otimes k$ canonically
identified with $\Cal H$:
\begin{equation}
x	=(\varepsilon\otimes I)(\Delta x)=bx,\hskip 15mm	
	x	=(I\otimes\varepsilon)(\Delta x)=ax,
\end{equation}
hence $a=b=1$. We shall use the following two variants of Sweedler's
notation:
\begin{eqnarray}
\Delta x	&=&\sum_{(x)}x_1\otimes x_2,	\\
	\wt\Delta x	&=&\sum_{(x)}x'\otimes x'',
\end{eqnarray}
the second being relevant only for $x\in\mop{Ker}\varepsilon$. if $x$ is
homogeneous of degree $n$ we can suppose that the components $x_1,x_2,x',x''$
in the expressions above are homogeneous as well, and we have then
$|x_1|+|x_2|=n$ and $|x'|+|x''|=n$ We easily compute:
\begin{eqnarray*}
(\Delta\otimes I)\Delta(x)	&=&x\otimes 1\otimes 1 +1\otimes x\otimes 1
						+1\otimes 1\otimes x	\\
				&\ +&\sum_{(x)}
x'\otimes x''\otimes 1 + x'\otimes 1\otimes x'' + 1\otimes x'\otimes x''\\
				&\ +&(\wt\Delta\otimes I)\wt \Delta (x)
\end{eqnarray*}
and
\begin{eqnarray*}
(I\otimes\Delta)\Delta(x)	&=&x\otimes 1\otimes 1 +1\otimes x\otimes 1
						+1\otimes 1\otimes x	\\
				&\ +&\sum_{(x)}
x'\otimes x''\otimes 1 + x'\otimes 1\otimes x'' + 1\otimes x'\otimes x''\\
				&\ +&(I\otimes\wt\Delta)\wt \Delta (x),
\end{eqnarray*}
hence the co-associativity of $\wt\Delta$ comes from the one of
$\Delta$. Finally it is easily seen by induction on $k$ that for any $x\in\Cal
H^n$ we can write:
\begin{equation}
\wt\Delta_k(x)=\sum_{x}x^{(1)}\otimes\cdots\otimes x^{(k+1)},
\end{equation}
with $|x^{(j)}|\ge 1$. The grading imposes:
$$\sum_{j=1}^{k+1}|x^{(j)}|=n,$$
so the maximum possible for any degree $|x^{(j)}|$ is $n-k$.
\end{proof}
\subsection{An example: the Hopf algebra of decorated rooted trees}
A \textsl{rooted tree} is an oriented graph with a finite number of vertices,
one among them being distinguished as the \textsl{root}, such that any vertex
admits exactly one incoming edge, except the root which has no incoming
edges. Here is the list of rooted trees up to five vertices:
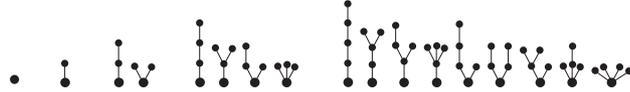
\begin{equation*}
\racine \hskip 5mm \arbrea \hskip 5mm  \arbreba \arbrebb \hskip 5mm  \arbreca \arbrecb \arbrecc
\arbrecd \hskip 5mm  \arbreda \arbredb \arbredc \arbredd \arbrede \arbredf \arbredz \arbredg
\arbredh
\end{equation*}
A \textsl{rooted forest} is a finite collection of rooted trees. The
Connes--Kreimer Hopf algebra ${\mathcal
  H}_{\makebox{{\tiny{CK}}}}=\bigoplus_{n\geq 0} {\mathcal
  H}_{\makebox{{\tiny{CK}}}}^{(n)}$ is the Hopf algebra of rooted forests over
$k$, graded by the number of vertices. It is the free commutative algebra on
the linear space $\Cal T$ spanned by nonempty rooted trees. The coproduct on a rooted forest $u$ (i.e. a product of rooted trees) is described as follows: the set $U$ of vertices of a forest $u$ is endowed with a partial order defined by $x \le y$ if and only if there is a path from a root to $y$ passing through $x$. Any subset $W$ of the set of vertices $U$ of $u$ defines a {\sl subforest\/} $w$ of $u$ in an obvious manner, i.e. by keeping the edges of $u$ which link two elements of $W$. The coproduct is then defined by:
\begin{equation}
\label{coprod}
	\Delta_{\makebox{{\tiny{CK}}}}(u)= \sum_{V \amalg W=U \atop W<V}v\otimes w.
\end{equation}
Here the notation $W<V$ means that $y<x$ for any vertex $x$ of $v$ and any
vertex $y$ of $w$ such that $x$ and $y$ are comparable. Such a couple $(V,W)$
is also called an {\sl admissible cut\/}, with crown (or pruning) $v$ and
trunk $w$. We have for example:
 \allowdisplaybreaks{
\begin{eqnarray*}
 \Delta_{\makebox{{\tiny{CK}}}}\big(\arbrea\big) &=&
                        \arbrea \otimes \un
                            + \un \otimes \arbrea
                               + \racine \otimes \racine \\
 \Delta_{\makebox{{\tiny{CK}}}}\big(\! \arbrebb \big) &=& 
 					\arbrebb \otimes \un
                               + \un \otimes \arbrebb +
                                  2\racine \otimes\arbrea
                                   + \racine\racine\otimes \racine 
\end{eqnarray*}}

With the restriction that $V$ and $W$ be nonempty (i.e. if $V$ and $W$ give rise to an ordered partition of $U$ into two blocks) we get the restricted coproduct:
\begin{equation}
\label{coprod2}
\wt\Delta_{\makebox{{\tiny{CK}}}}(u)=\Delta_{\makebox{{\tiny{CK}}}}(u)-u\otimes\un -\un\otimes u=\sum_{V\amalg W=U \atop W<V,\, V,W\not =\emptyset}v\otimes w,
\end{equation}
which is often displayed $\sum_{(u)} u' \otimes u''$ in Sweedler's notation. The iterated restricted coproduct writes in terms of ordered partitions of $U$ into $n$ blocks:
\begin{equation}
\label{iter-coprod}
\wt\Delta_{\makebox{{\tiny{CK}}}}^{n-1}(u)=\sum_{V_1\amalg\cdots\amalg V_n=U \atop V_n<\cdots <V_1,\,
  V_j\not =\emptyset}v_1\otimes\cdots\otimes v_n,
\end{equation}
and we get the full iterated coproduct
$\Delta_{\makebox{{\tiny{CK}}}}^{n-1}(u)$ by allowing empty blocks in the
formula above. Coassociativity of the coproduct follows immediately.\\

Note however that the relation $<$ on subsets of vertices is not transitive. The notation $V_n<\cdots <V_1$ is to be understood as $V_i<V_j$ for any $i>j,\  i,j\in\{1,\ldots ,n\}$.
\subsection{Connected filtered bialgebras}
A {\sl filtered bialgebra\/} on $k$ is a $k$-vector space together with an increasing $\N$-indexed filtration:
$$\Cal H^0\subset\Cal H^1\subset\cdots\subset \Cal H^n\subset\cdots, \bigcup_n \Cal H^n=\Cal H$$
endowed with a product $m:\Cal H\otimes \Cal H\to\Cal H$, a coproduct
$\Delta:\Cal H\to\Cal H\otimes \Cal H$, a unit $u:k\to\Cal H$,
a co-unit $\varepsilon:\Cal H\to k$ and an antipode $S:\Cal H\to\Cal H$
fulfilling the usual axioms of a bialgebra, and such that:
\begin{eqnarray}
\Cal H^p.\Cal H^q	&\subset& \Cal H^{p+q}	\\
		\Delta(\Cal H^n)	&\subset& \sum_{p+q=n}\Cal H^p\otimes\Cal H^q.
\end{eqnarray}
It is easy (and left to the reader) to show that the unit $u$ and the co-unit $\varepsilon$ are of degree zero, if we consider the filtration on the base field $k$ given by $k^0=\{0\}$ and $k^n=k$ for any $n\ge 1$. Namely, $u(k^n)\subset \Cal H^n$ and $\varepsilon(\Cal H^n)\subset k^n$ for any $n\ge 0$.\\

If we ask for the existence of an antipode $S$ we get the
definition of a {\sl filtered Hopf algebra\/} if the antipode is of degree zero i.e. if:
\begin{equation}
S(\Cal H^n)\subset \Cal H^n
\end{equation}
for any $n\ge 0$. Contrarily to the graded case, it is not likely that a filtered bialgebra with antipode is automatically a filtered Hopf algebra (see e.g. \cite[Lemma 5.2.8]{Mo93}, \cite{AS02} and \cite{AC11}). The antipode is however of degree zero in the connected case:\\

For any $x\in\Cal H$ we set:
\begin{equation}
|x|:=\mop{min}\{n\in\N,\ x\in\Cal H^n\}.
\end{equation}
Any graded bialgebra or Hopf algebra is obviously filtered by the canonical
filtration associated to the grading:
\begin{equation}
\Cal H^n:=\bigoplus_{i=0}^n \Cal H_i,
\end{equation}
and in that case, if $x$ is a homogeneous element, $x$ is of degree $n$ if
and only if $|x|=n$. We say that the filtered bialgebra $\Cal H$ is connected
if $\Cal H^0$ is one-dimensional. There is an analogue of proposition \ref{coproduit-connexe}
in the connected filtered case, the proof of which is very similar:
\begin{prop}
For any $x\in\Cal H^n, n\ge 1$ we can write:
\begin{equation}
\Delta x=x\otimes\un+\un\otimes x+\wt\Delta x,\hskip 15mm 
\wt\Delta x\in\sum_{p+q=n,\, p\not= 0,\, q\not= 0}\Cal H^p\otimes\Cal H^q.
\end{equation}
The map $\wt\Delta$ is coassociative on $\mop{Ker}\varepsilon$ and $\wt\Delta_k=(I^{\otimes k-1}\otimes\wt\Delta)(I^{\otimes k-2}\otimes\wt\Delta)...\wt\Delta$ sends $\Cal H^n$ into $(\Cal H^{n-k})^{\otimes k+1}$.
\end{prop}
As an easy corollary, the degree of the antipode is also zero in the connected case, i.e. $S(\Cal H^n)\subseteq \Cal H^n$ for any $n$. This is an immediate consequence of the recursive formulae \eqref{antipode1} and \eqref{antipode2} below. 
\subsection{The convolution product}\label{convolution}
An important result is that any connected filtered bialgebra is indeed a
filtered Hopf algebra, in the sense that the antipode comes for free. We give
a proof of this fact as well as a recursive formula for the antipode with the
help of the {\sl convolution product\/}: let $\Cal H$ be a (connected filtered) bialgebra, and let $\Cal A$ be any
$k$-algebra (which will be called the {\sl target algebra\/}): the
convolution product on $\Cal L(\Cal H,\Cal A)$ is given by:
\begin{eqnarray*}
\varphi*\psi (x)	&=&m_{\Cal A}(\varphi\otimes\psi)\Delta(x)\\
				&=&\sum_{(x)}\varphi(x_1)\psi(x_2).
\end{eqnarray*}
\begin{prop}\label{groupeG}
The map $e=u_{\Cal A}\circ\varepsilon$, given by $e(\un)=\un_{\Cal A}$ and $e(x)=0$ for any $x\in\mop{Ker}\varepsilon$, is a unit for the convolution product. Moreover the set $G(\Cal A):=\{\varphi\in\Cal L(\Cal H,\Cal A),\ \varphi(\un)=\un_{\Cal A}\}$ endowed with the convolution product is a group.
\end{prop}
\begin{proof}
The first statement is straightforward. To prove the second let us consider
the formal series:
\begin{eqnarray*}
\varphi^{*-1}(x)	&=&\big(e-(e-\varphi)\big)^{*-1}(x)	\\
				&=&\sum_{k\ge 0}(e-\varphi)^{*k}(x).
\end{eqnarray*}
Using $(e-\varphi)(\un)=0$ we have immediately $(e-\varphi)^{*k}(\un)=0$, and
for any $x\in\mop{Ker}\varepsilon$:
\begin{equation}
(e-\varphi)^{*k}(x)=(-1)^km_{\Cal
  A,k-1}(\underbrace{\varphi\otimes\cdots\otimes\varphi}_{k\hbox{ \sevenrm times }})\wt\Delta_{k-1}(x).
\end{equation}
When $x\in\Cal H^n$ this expression vanishes then for $k\ge n+1$. The formal series ends up then with a finite number of terms for any $x$, which proves the result.
\end{proof}
\begin{cor}\label{antipode}
Any connected filtered bialgebra $\Cal H$ is a filtered Hopf algebra. The
antipode is defined by:
\begin{equation}\label{eq:antipode}
S(x)=\sum_{k\ge 0}(u\varepsilon-I)^{*k}(x).
\end{equation}
It is given by $S(\un)=\un$ and recursively by any of the two formulae for
$x\in\mop{Ker}\varepsilon$:
\begin{eqnarray}
S(x)	&=-x-\sum_{(x)}S(x')x''\label{antipode1}	\\
		S(x)	&=-x-\sum_{(x)}x'S(x'').\label{antipode2}
\end{eqnarray}
\end{cor}
\begin{proof}
The antipode, when it exists, is the inverse of the identity for the convolution product on $\Cal L(\Cal H,\Cal H)$. One just needs then to apply Proposition \ref{groupeG} with $\Cal A=\Cal H$. The two recursive formulas come directly from the two equalities:
$$m(S\otimes I)\Delta (x)=m(I\otimes S)\Delta (x)=0$$
fulfilled by any $x\in\mop{Ker}\varepsilon$.
\end{proof}
Let $\g g(\Cal A)$ be the subspace of $\Cal L(\Cal H,\Cal A)$ formed by the elements $\alpha$ such that $\alpha(\un)=0$. It is clearly a subalgebra of $\Cal L(\Cal H,\Cal A)$ for the convolution product. We have:
$$G(\Cal A)=e+\g g(\Cal A).$$
From now on we shall suppose that the ground field $k$ is of characteristic
zero. For any $x\in\Cal H^n$ the exponential:
\begin{equation}
e^{*\alpha}(x)=\sum_{k\ge 0}{\alpha^{*k}(x)\over k!}
\end{equation}
is a finite sum (ending up at $k=n$). It is a bijection from $\g g(\Cal A)$
onto $G(\Cal A)$. Its inverse is given by:
\begin{equation}
\mop{Log}(1+\alpha)(x)=\sum_{k\ge 1}{(-1)^{k-1}\over k}\alpha^{*k}(x).
\end{equation}
This sum again ends up at $k=n$ for any $x\in\Cal H^n$. Let us introduce a
decreasing filtration on $\Cal L=\Cal L(\Cal H,\Cal A)$:
\begin{equation}
\Cal L^n:=\{\alpha\in\Cal L, \ \alpha\restr{\Cal H^{n-1}}=0\}.
\end{equation}
Clearly $\Cal L_0=\Cal L$ and $\Cal L_1=\g g(\Cal A)$. We define the valuation
$\mop{val}\varphi$ of an element $\varphi$ of $\Cal L$ as the greatest integer
$k$ such that $\varphi$ is in $\Cal L_k$. We shall consider in the sequel the
ultrametric distance on $\Cal L$ induced by the filtration:
\begin{equation}\label{distance}
d(\varphi,\psi)=2^{-\smop{val}(\varphi-\psi)}.
\end{equation}
For any $\alpha,\beta\in\g g(\Cal A)$ let $[\alpha,\beta]=\alpha*\beta-\beta*\alpha$.
\begin{prop}
We have the inclusion:
\begin{equation}
\Cal L^p*\Cal L^q\subset\Cal L^{p+q},
\end{equation}
and moreover the metric space $\Cal L$ endowed with the distance defined by \eqref{distance} is complete.
\end{prop}
\begin{proof}
Take any $x\in\Cal H^{p+q-1}$, and any $\alpha\in\Cal L_p$ and $\beta\in\Cal L_q$. We have:
$$(\alpha*\beta)(x)=\sum_{(x)}\alpha(x_1)\beta(x_2).$$
Recall that we denote by $|x|$ the minimal $n$ such that $x\in\Cal H^n$. Since $|x_1|+|x_2|=|x|\le p+q-1$, either $|x_1|\le p-1$ or $|x_2|\le q-1$, so the expression vanishes. Now if $(\psi_n)$ is a Cauchy sequence in $\Cal L$ it is immediate to see that this sequence is {\sl locally stationary\/}, i.e. for any $x\in\Cal H$ there exists $N(x)\in\N$ such that $\psi_n(x)=\psi_{N(x)}(x)$ for any $n\ge N(x)$. Then the limit of $(\psi_n)$ exists and is clearly defined by:
$$\psi(x)=\psi_{N(x)}(x).$$  
\end{proof}
As a corollary the Lie algebra $\Cal L_1=\g g(\Cal A)$ is {\sl pro-nilpotent},
in a sense that it is the projective limit of the Lie algebras $\g g(\Cal
A)/\Cal L^n$, which are nilpotent.
\subsection{Characters}\label{sect:characters}
Let $\Cal H$ be a connected filtered Hopf algebra over $k$, and let $\Cal A$ be a $k$-algebra. We shall consider unital algebra morphisms from $\Cal H$ to the target algebra $\Cal A$. When the algebra $\Cal A$ is commutative we shall call them slightly abusively {\sl characters\/}. We recover of course the usual notion of character when the algebra $\Cal A$ is the ground field $k$.
\\

The notion of character involves only the algebra structure of
$\Cal H$. On the other hand the convolution product on $\Cal L(\Cal H,\Cal A)$
involves only the {\sl coalgebra\/} structure on $\Cal H$. Let us consider now
the full Hopf algebra structure on $\Cal H$ and see what happens to algebra
morphisms with the convolution product:
\begin{prop}\label{prop:convolution2}
Let $\Cal H$ be any Hopf algebra over $k$, and let $\Cal A$
be a \textsl{commutative} $k$-algebra. Then the characters from $\Cal
  H$ to $\Cal A$ form a group $G_1(\Cal A)$ under the convolution product, and
  for any $\varphi\in G_1(\Cal A)$ the inverse is given by:
\begin{equation}
\varphi^{*-1}=\varphi\circ S.
\end{equation}
\end{prop}
\begin{proof}
Using the fact that $\Delta$ is an algebra morphism we have for any $x,y\in\Cal H$:
$$f*g(xy)=\sum_{(x)(y)}f(x_1y_1)g(x_2y_2).$$
If $\Cal A$ is commutative and if $f$ and $g$ are characters we get:
\begin{eqnarray*}
f*g(xy)	&=&\sum_{(x)(y)}f(x_1)f(y_1)g(x_2)g(y_2)	\\
			&=&\sum_{(x)(y)}f(x_1)g(x_2)f(y_1)g(y_2)	\\
			&=&(f*g)(x)(f*g)(y).
\end{eqnarray*}
The unit $e=u_{\Cal A}\circ\varepsilon$ is an algebra morphism. The formula for the inverse of a character comes easily from the commutativity of the following diagram:
\diagramme{
\xymatrix{&\Cal H\otimes\Cal H	\ar[rr]^{S\otimes I}
				&&\Cal H\otimes \Cal H\ar[dr]^{m}\ar[rr]^{f\otimes f}	&& \Cal\otimes\Cal A	\ar[dr]^{m_{\Cal A}}	&\\
\Cal H\ar[rr]_\varepsilon \ar[dr]^\Delta \ar[ur]^\Delta
	\ar@/_4.5pc/	@{-->}[rrrrrr]_{f*(f\circ S)}
	\ar@/^4.5pc/	@{-->}[rrrrrr]^{(f\circ S)*f}
	\ar@/^1pc/	@{-->}[rrrrrr]_e
				&&	k\ar[rr]_u	&&\Cal H\ar[rr]_f
	&&	\Cal A	\\
&\Cal H\otimes\Cal H	\ar[rr]^{I\otimes S}
				&&\Cal H\otimes \Cal H\ar[ur]^{m}
\ar[rr]^{f\otimes f}	&&\Cal A\otimes\Cal A\ar[ur]^{m_\Cal A}	&}
}
\end{proof}
We call {\sl  infinitesimal characters with values in the algebra $\Cal A$\/} those elements $\alpha$ of $\Cal L(\Cal H,\Cal A)$ such that:
$$\alpha(xy)=e(x)\alpha(y)+\alpha(x)e(y).$$
\begin{prop}\label{prop:exp}
Let $\Cal H$ be a connected filtered Hopf algebra, and suppose that $\Cal A$ is a commutative algebra.
Let $G_1(\Cal A)$ (resp. $\g g_1(\Cal A)$) be the set of characters of $\Cal H$ with values in
$\Cal A$ (resp the set of infinitesimal characters of $\Cal H$ with values in
$\Cal A$). Then $G_1(\Cal A)$ is a subgroup of $G$, the exponential restricts to a
bijection from $\g g_1(\Cal A)$ onto $G_1(\Cal A)$, and $\g g_1(\Cal A)$ is a Lie subalgebra of $\g g(\Cal A)$.
\end{prop}
\begin{proof}
Take two infinitesimal characters $\alpha$ and $\beta$ with values in
$\Cal A$ and compute:
\begin{eqnarray*}
(\alpha*\beta)(xy)
	&=&\sum_{(x)(y)}\alpha(x_1y_1)\beta(x_2y_2)	\\
	&=&\sum_{(x)(y)}\big(\alpha(x_1)e(y_1)+e(x_1)\alpha(y_1)\big).
		\big(\beta(x_2)e(y_2)+e(x_2)\alpha(y_2)\big)		\\
	&=&(\alpha*\beta)(x)e(y)+\alpha(x)\beta(y)+\beta(x)\alpha(y)
		+e(x)(\alpha*\beta)(y).
\end{eqnarray*}
Using the commutativity of $\Cal A$ we immediately get:
$$[\alpha,\beta](xy)=[\alpha,\beta](x)e(y)+e(x)[\alpha,\beta](y),$$
which shows that $\g g_1(\Cal A)$ is a Lie algebra. Now for $\alpha\in\g g_1(\Cal A)$ we have:
$$\alpha^{*n}(xy)=\sum_{k=0}^n{n\choose k}\alpha^{*k}(x)\alpha^{*(n-k)}(y),$$
as easily seen by induction on $n$. A straightforward computation then yields:
$$\exp(\alpha)(xy)=\exp(\alpha)(x)\exp(\alpha)(y),$$
with
$$\exp\alpha:=\sum_{k\ge 0} \frac{\alpha^{*k}}{k!}=e+\alpha+\frac{\alpha^{*2}}{2}+\cdots .$$
The series above makes sense thanks to connectedness, as explained in Paragraph \ref{convolution}. Now let $\varphi=e+\gamma\in G_1(\Cal A)$, and let $\log(\varphi)=\sum_{j\ge 1}\frac{(-1)^{j-1}\gamma^{*j}}{j}$. Set $\varphi^{*t}:=\exp(t\log \varphi)$ for $t\in k$. It coincides with the $n^{\smop{th}}$ convolution power of $\varphi$ for any integer $n$. Hence $\varphi^{*t}$ is an $\Cal A$-valued character of $\Cal H$ for any $t\in k$. Indeed, for any $x,y\in \Cal H$ the expression $\varphi^{*t}(xy)-\varphi^{*t}(x)\varphi^{*t}(y)$ is polynomial in $t$ and vanishes on all integers, hence vanishes identically. Differentiating with respect to $t$ at $t=0$ we immediately find that $\log \varphi$ is an infinitesimal character.
\end{proof}
\subsection{Group schemes and the Cartier-Milnor-Moore-Quillen theorem}\label{sect:cmm}
\begin{thm}[Cartier, Milnor, Moore, Quillen]
Let $\Cal U$ be a cocommutative connected filtered Hopf algebra and let $\g g$ be
the Lie algebra of its primitive elements, endowed with the filtration induced by the one of $\Cal U$, which in turns induces a filtration on the enveloping algebra $\Cal U(\g g)$. Then $\Cal U$
and $\Cal U(\g g)$ are isomorphic as filtered Hopf
algebras. If $\Cal U$ is moreover graded, then the two Hopf algebras are
isomorphic as graded Hopf algebras.
\end{thm}
\begin{proof}
The following proof is borrowed from L. Foissy's thesis. The embedding $\iota:\g g\to\Cal U$ obviously induces an algebra morphism
\begin{equation}
\varphi:\Cal U(\g g)\longrightarrow \Cal U.
\end{equation}
It is easy to show that $\varphi$ is also a coalgebra morphism. It remains to
show that $\varphi$ is surjective, injective, and respects the
filtrations. Let us first prove the surjectivity by induction on the coradical
filtration degree: 
\begin{equation}
d(x):=\mop{min}\{n\in\N,\ \widetilde\Delta^{n}(x)=0\}.
\end{equation}
Set $\Cal U^n:=\{x\in\Cal U,\ d(x)\le n\}$, and similarly for $\Cal U(\g g)$. We can limit ourselves to the kernel of the co-unit. Any
$x\in\Cal U^1\cap \mop{Ker}\varepsilon$ is primitive, hence $\varphi:\Cal U(\g
g)^1\to\Cal U^1$ is obviously a linear isomorphism. Now for $x\in\Cal U^n\cap \mop{Ker}\varepsilon$
(for some integer $n\ge 2$) we can write, using cocommutativity:
\begin{eqnarray*}
\widetilde\Delta^{n-1}(x)&=&\sum_{(x)}x^{(1)}\otimes\cdots\otimes x^{(n)}\\
&=&\frac {1}{n!}\sum_{\sigma\in S_n}\sum_{(x)}x^{(\sigma_1)}\otimes\cdots\otimes x^{(\sigma_n)}
\end{eqnarray*}
where the $x^{(j)}$'s are of coradical filtration degree $1$, hence primitive. But we also have:
\begin{equation}
\widetilde\Delta^{n-1}(x^{(1)}\cdots x^{(n)})=\sum_{\sigma\in S_n}x^{(\sigma_1)}\otimes\cdots\otimes x^{(\sigma_n)}.
\end{equation}
Hence the element $y=x-\frac{1}{n!}\sum_{(x)}x^{(1)}\otimes\cdots\otimes
x^{(n)}$ belongs to $\Cal U^{n-1}$. It is a linear combination of
products of primitive elements by induction hypothesis, hence so is $x$. We
have thus proved that $\Cal U$ is generated by $\g g$, which amounts to the
surjectivity of $\varphi$.\\

Now consider a nonzero element $u\in \Cal U(\g g)$ such that $\varphi(u)=0$,
and such that $d(u)$ is minimal. We have already proved $d(u)\ge 2$. We now
compute:
\begin{eqnarray*}
0=\Delta\big(\varphi(u)\big)&=&(\varphi\otimes\varphi)\Delta(u)\\
&=&(\varphi\otimes\varphi)\left(u\otimes 1+1\otimes u+\sum_{(x)}u'\otimes
  u''\right)\\
&=&\sum_{(u)}\varphi(u')\otimes\varphi(u'').
\end{eqnarray*}
By minimality hypothesis on $d(u)$, we get then $\sum_{(u)}u'\otimes
u''=0$. Hence $u$ is primitive, i.e. $d(u)=1$, a contradiction. Hence
$\varphi$ is injective. The compatibility with the original filtration or
graduation is obvious.
\end{proof}
Now let $\Cal H:\bigcup_{n\ge 0}\Cal H^n$ be a connected filtered Hopf algebra and let $\Cal A$ be
a commutative unital algebra. We suppose that the components of the filtration
are finite-dimensional. The group $G_1(\Cal A)$ defined in the previous
paragraph depends functorially on the target algebra
$\Cal A$: In particular, when the Hopf algebra $\Cal H$ itself is commutative,
the correspondence $\Cal A\mapsto G_1(\Cal A)$ is a {\sl group
  scheme\/}. In the graded case with finite-dimensional components, it is possible to reconstruct the Hopf algebra $\Cal
H$ from the group scheme. We have indeed:
\begin{prop}
\begin{equation}
\Cal H=\Big(\Cal U\big(\g g_1(k)\big)\Big)^\circ,
\end{equation}
where $\g g_1(k)$ is the Lie algebra of infinitesimal characters with values
in the base field $k$, where $\Cal U\big(\g g_1(k)\big)$ stands for its enveloping
algebra, and where $(-)^\circ$ stands for the graded dual.
\end{prop}
In the case when the Hopf algebra $\Cal H$ is not commutative this is no longer possible to
reconstruct it from $G_1(k)$.
\subsection{Renormalization in connected filtered Hopf algebras}\label{sect:ren}
We describe in this section the renormalization \`a la Connes-Kreimer
(\cite{K}, \cite{CK1}) in the
abstract context of connected filtered Hopf algebras: the objects to be
renormalized are characters with values in a commutative unital target algebra
$\Cal A$ endowed with a {\sl renormalization scheme\/}, i.e. a splitting $\Cal
A=\Cal A_-\oplus\Cal A_+$ into two subalgebras. An important example is given
by the {\sl minimal subtraction scheme\/} of the algebra $\Cal A$ of meromorphic functions of one variable $z$, where
$\Cal A_+$ is the algebra of meromorphic functions which are holomorphic at
$z=0$, and where $\Cal A_-=z^{-1} \CC [z^{-1}]$ stands for the ``polar parts''. Any $\Cal A$-valued character $\varphi$ admits a unique {\sl Birkhoff decomposition\/}:
$$\varphi=\varphi_-^{*-1}*\varphi_+,$$
where $\varphi_+$ is an $\Cal A_+$-valued character, and where
$\varphi(\mop{Ker}\varepsilon)\subset \Cal A_-$. In the MS scheme case
described just above, the renormalized character is the scalar-valued
character given by the evaluation of $\varphi_+$ at $z=0$ (whereas the
evaluation of $\varphi$ at $z=0$ does not necessarily make sense).\\

We consider here the situation where the algebra $\Cal A$ admits a {\sl
renormalization scheme\/}, i.e. a splitting into two subalgebras:
$$\Cal A=\Cal A_-\oplus \Cal A_+$$
with $\un\in\Cal A_+$. Let $\pi:\Cal A\surj{6}\Cal A_-$ be the projection on
$\Cal A_-$ parallel to $\Cal A_+$.
\begin{thm}
\begin{enumerate}
\item Let $\Cal H$ be a connected filtered Hopf algebra. Let $G(\Cal A)$ be the group
  of those $\varphi\in\Cal L(\Cal H,\Cal A)$ such that $\varphi(\un)=\un_{\Cal
    A}$ endowed with the convolution product. Any $\varphi\in G(\Cal A)$ admits a
  unique Birkhoff decomposition:
\begin{equation}
\varphi=\varphi_-^{*-1} * \varphi_+,
\end{equation}
where $\varphi_-$ sends $\un$ to $\un_{\Cal A}$ and $\mop{Ker}\varepsilon$
into $\Cal A_-$, and where $\varphi_+$ sends $\Cal H$ into $\Cal A_+$. The maps
$\varphi_-$ and $\varphi_+$ are given on $\mop{Ker}\varepsilon$ by the
following recursive formulas:
\begin{eqnarray}
\varphi_-(x)	&=&-\pi\Big( \varphi(x)+\sum_{(x)}\varphi_-(x')\varphi(x'')\Big)	\\
\varphi_+(x)	&=&(I-\pi)\Big(
\varphi(x)+\sum_{(x)}\varphi_-(x')\varphi(x'')\Big).
\end{eqnarray}
\item If the algebra $\Cal A$ is commutative and if $\varphi$ is a character, the components $\varphi_-$ and $\varphi_+$ occurring in the Birkhoff decomposition of $\chi$ are characters as well. 
\end{enumerate}
\end{thm}
\begin{proof}
The proof goes along the same lines as the proof of Theorem 4 of \cite
{CK1}: for the first assertion it is immediate from the definition of $\pi$
that $\varphi_-$ sends $\mop{Ker}\varepsilon$ into $\Cal A_-$, and that
$\varphi_+$ sends $\mop{Ker}\varepsilon$ into $\Cal A_+$. It only remains to
check equality $\varphi_+=\varphi_-*\varphi$, which is an easy computation:
\begin{eqnarray*}
\varphi_+(x)	&=&(I-\pi)\Bigl( \varphi(x)+\sum_{(x)}\varphi_-(x')\varphi(x'')\Bigr).	\\
			&=&\varphi(x)+\varphi_-(x)+ \sum_{(x)}\varphi_-(x')\varphi(x'')	\\
			&=&(\varphi_-*\varphi)(x).
\end{eqnarray*}

The proof of assertion 2) goes exactly as in \cite {CK1} and relies on the
following {\sl Rota-Baxter relation\/} in $\Cal A$:
\begin{equation}
\pi(a)\pi(b)=\pi\big(\pi(a)b+a\pi(b)\big)-\pi(ab),
\end{equation}
which is easily verified by decomposing $a$ and $b$ into their $\Cal A_\pm$-parts. Let $\varphi$ be a character of $\Cal H$ with values in $\Cal A$. Suppose that we have $\varphi_-(xy)=\varphi_-(x)\varphi_-(y)$ for any $x,y\in\Cal H$ such that $|x|+|y|\le d-1$, and compute for $x,y$ such that $|x|+|y|=d$:
$$\varphi_-(x)\varphi_-(y)=\pi(X)\pi(Y),$$
with $X=\varphi(x)-\sum_{(x)}\varphi_-(x')\varphi(x'')$ and $Y=\varphi(y)-\sum_{(y)}\varphi_-(y')\varphi(y'')$. Using the formula:
$$\pi(X)=-\varphi_-(x),$$
we get:
$$\varphi_-(x)\varphi_-(y)=-\pi\bigl(XY+\varphi_-(x)Y+X\varphi_-(y)\bigr),$$
hence:
\begin{eqnarray*}
\varphi_-(x)\varphi_-(y)&=&-\pi\Big(\varphi(x)\varphi(y)+\varphi_-(x)\varphi(y)+\varphi(x)\varphi_-(y)\\
&{}&\hskip 8mm+\sum_{(x)}\varphi_-(x')\varphi(x'')\big(\varphi(y)+\varphi_-(y)\big)
	+\sum_{(y)}\big(\varphi(x)+\varphi_-(x)\big)\varphi_-(y')\varphi(y'')\\
&{}&\hskip 8mm+\sum_{(x)(y)}\varphi_-(x')\varphi(x'')\varphi_-(y')\varphi(y'')\Big).
\end{eqnarray*}
We have to compare this expression with:
\begin{eqnarray*}
\varphi_-(xy)&=&-\pi\Big(\varphi(xy)+\varphi_-(x)\varphi(y)+\varphi_-(y)\varphi(x)\\
&{}&\hskip 8mm+\sum_{(x)}\big(\varphi_-(x'y)\varphi(x'')+\varphi_-(x')\varphi(x''y) \big)
+\sum_{(y)}\big(\varphi_-(xy')\varphi(y'')+\varphi_-(y')\varphi(xy'') \big) \\
&{}&\hskip 8mm+\sum_{(x)(y)}\varphi_-(x'y')\varphi(x''y'')	\Big).
\end{eqnarray*}
These two expressions are easily seen to be equal using the commutativity of the algebra $\Cal A$, the character property for $\varphi$ and the induction hypothesis.
\end{proof}
\begin{rmk}
Assertion 2) admits a more conceptual proof (see the notes by K. Ebrahimi-Fard
in the present volume), which is based on the following recursive expressions
for the components of the Birkhoff decomposition: define the {\rm Bogoliubov preparation map\/} as the map $b:G(\Cal A)\to \Cal
L(\Cal H,\Cal A)$ recursively given by:
\begin{equation}
b(\varphi)(x)=\varphi(x)+\sum_{(x)}\varphi_-(x')\varphi(x'').
\end{equation}
Then the components of $\varphi$ in the Birkhoff decomposition read:
\begin{equation}\label{bogo2}
\varphi_-=-\pi\circ b(\varphi),\hskip 12mm \varphi_+=(I-\pi)\circ
b(\varphi).
\end{equation}
Bogoliubov preparation map also writes in more concise form:
\begin{equation}\label{bogo3}
b(\varphi)=\varphi_-*(\varphi-e).
\end{equation}
Plugging equation (\ref{bogo3}) inside (\ref{bogo2}) and setting
$\alpha:=e-\varphi$ we get the following expression for $\varphi_-$:
\begin{eqnarray}
\varphi_-&=&e+P(\varphi_-*\alpha)\\
&=&e+P(\alpha)+P\big(P(\alpha)*\alpha\big)+\cdots+
\underbrace{P\Big(P\big(...P(}_{\hbox{\sevenrm $\scriptstyle n$ times}}\alpha)*\alpha\big)\cdots *\alpha\Big)+\cdots\label{pre-spitzer}
\end{eqnarray}
where $P:\Cal L(\Cal H,\Cal A)\to\Cal L(\Cal H,\Cal A)$ is the projection
defined by $P(\alpha)=\pi\circ \alpha$. The renormalized part $\varphi_+$
satisfies an analogous recursive expression:
\begin{align}
	\varphi_+	&=	e+\tilde{P}(\varphi_-*\alpha) 						\label{pre-spitzer+} \\
				&= e+\tilde P(\varphi_+*\beta)\\
			&=	e+\tilde{P}(\beta)+\tilde{P}\bigl(\tilde{P}(\beta) * \beta\bigr) + \cdots +
				\underbrace{\tilde{P}\Bigl(\tilde{P}\bigl(\ldots \tilde{P}(}_{\hbox{\sevenrm $\scriptstyle n$ times}} 	  
								\beta)*\beta\bigr)\cdots *\beta\Bigr)+ \cdots     \nonumber   
\end{align}
with $\beta:=\varphi^{-1}*\alpha=e-\varphi^{-1}$, and where $\tilde{P}=I-P$ is the projection on $\mathcal L(\mathcal H,\mathcal A)$ defined by $\tilde{P}(\alpha)=(I-\pi) \circ \alpha$.
\end{rmk}
%%%%%
\section{Pre-Lie algebras}
%%%%%
Pre-Lie algebras are sometimes called {\sl Vinberg algebras\/}, as they appear
in the work of E. B. Vinberg \cite{V63} under the name ``left-symmetric
algebras'' on the classification of
homogeneous cones. They appear independently at the same time in the work of
M. Gerstenhaber \cite{G63} on Hochschild cohomology and
deformations of algebras, under the name ``pre-Lie algebras'' which is now the
standard terminology. The term ``chronological algebras'' has also been
sometimes used, e.g. in the fundamental work of A. Agrachev and R. Gamkrelidze
\cite{AG81}. The notion itself can be however traced back to
the work of A. Cayley \cite{Cay} which, in modern
language, describes {\sl the\/} pre-Lie algebra morphism $F_a$ from the pre-Lie
algebra of rooted trees into the pre-Lie algebra of vector fields on $\R^n$
sending the one-vertex tree to a given vector field $a$. For a survey
emphasizing on geometric aspects, see \cite{Bu06}.
\subsection{Definition and general properties}
A {\sl left pre-Lie algebra\/} over a field $k$ is a $k$-vector space $A$ with a bilinear binary composition $ \rhd  $ that satisfies the left pre-Lie identity:
\begin{equation}
    (a \rhd   b) \rhd   c-a \rhd  (b \rhd  c)=
    (b \rhd   a) \rhd   c-b \rhd  (a \rhd   c),
    \label{prelie}
\end{equation}
for $a,b,c \in A$. Analogously, a {\sl right pre-Lie
algebra\/} is a $k$-vector space $A$ with a binary composition $ \lhd  $ that satisfies the right pre-Lie identity:
\begin{equation}
    (a \lhd   b) \lhd   c-a \lhd  (b \lhd  c)=
    (a \lhd   c) \lhd   b-a \lhd  (c \lhd   b).
    \label{prelie2}
\end{equation} 
The left pre-Lie identity rewrites as:
\begin{equation}\label{prelie1-2}
L_{[a,b]}=[L_a,L_b],
\end{equation}
where $L_a:A\to A$ is defined by $L_ab=a\rhd b$, and where the bracket on the
left-hand side is defined by $[a,b]:=a\rhd b-b\rhd a$. As an easy consequence
this bracket satisfies the Jacobi identity: If $A$ is unital (i.e. there
exists $\un\in A$ such that $\un\rhd a=a\rhd 1=1$ for any $a\in A$) it is
immediate thanks to the fact that $L: A\to\mop{End}A$ is injective. If not, we
can add a unit by considering $\overline A:=A\oplus k.\un$ and extend
$L:\overline A\to\mop{End}\overline A$ accordingly. As any right pre-Lie algebra $(A,\lhd)$ is also a left pre-Lie algebra with
product $a\rhd b:=b\lhd a$, one can stick to left pre-Lie algebras, what we shall do unless specifically indicated.
\subsection{The group of formal flows}\label{sect:formal-flows}
The following is taken from the paper of A. Agrachev and R. Gamkrelidze
\cite{AG81}. Suppose that $A$ is a left pre-Lie algebra endowed with a
compatible decreasing filtration, namely $A=A_1\supset A_2\subset A_3\supset\cdots$, such that the intersection of the $A_j$'s reduces to $\{0\}$, and such that $A_p\rhd A_q\subset A_{p+q}$. Suppose moreover that $A$ is complete with respect to this filtration. The Baker-Campbell-Hausdorff formula:
\begin{equation}
C(a,b)=a+b+\frac 12[a,b]+\frac 1{12}([a,[a,b]]+[b,[b,a]])+\cdots
\end{equation}
endows then $A$ with a structure of pro-unipotent group. An example of this
situation is given by $A=hB[[h]]$ where $B$ is any pre-Lie algebra, and
$A_j=h^jB[[h]]$. This group admits a
more transparent presentation as follows : introduce a fictitious unit $\un$
such that $\un\rhd a=a\rhd\un=a$ for any $a\in A$, and define $W:A\to A$ by:
\begin{equation}
W(a):=e^{L_a}\un-\un=a+\frac 12 a\rhd a+\frac 16 a\rhd(a\rhd a)+\cdots.
\end{equation}
The application $W$ is clearly a bijection. The inverse, denoted by $\Omega$, also appears under the name "pre-Lie Magnus expansion" in \cite{EM}. It verifies the equation:
\begin{equation}
\Omega(a)=\frac{L_{\Omega(a)}}{e^{L_{\Omega(a)}}-Id}a=\sum_{i\ge 0}B_iL_{\Omega(a)}^{i}a,
\end{equation}
where the $B_i$'s are the Bernoulli numbers. The first few terms are:
\begin{equation}
\Omega(a)=a-\frac 12 a\rhd a +\frac 14 (a\rhd a)\rhd a+\frac 1{12} a\rhd(a\rhd a)+\cdots
\end{equation}
Transferring the BCH product by means of the map $W$, namely:
\begin{equation}\label{diese}
a\# b=W\Big(C\big(\Omega(a),\Omega(b)\big)\Big),
\end{equation}
we have $W(a)\#W(b)=W\big(C(a,b)\big)=e^{L_a}e^{L_b}\un-\un$, hence $W(a)\#W(b)=W(a)+e^{L_a}W(b)$. The product $\#$ is thus given by the simple formula:
\begin{equation}
a\# b=a+e^{L_{\Omega(a)}}b.
\end{equation}
The inverse is given by
$a^{\#-1}=W\big(-\Omega(a)\big)=e^{-L_{\Omega(a)}}\un-\un$. If $(A,\rhd)$ and
$(B,\rhd)$ are two such pre-Lie algebras and $\psi:A\to B$ is a
filtration-preserving pre-Lie algebra morphism, it is immediate to check that
for any $a,b\in A$ we have:
\begin{equation}
\psi(a\# b)=\psi(a)\#\psi(b).
\end{equation}
In other words, the group of formal flows is a functor from the category of
complete filtered pre-Lie algebras to the category of groups.\\

When the pre-Lie product $\rhd$ is associative, all this simplifies to:
\begin{equation}
a\#b=a\rhd
b+a+b
\end{equation}
and
\begin{equation}
a^{\#-1}=\frac 1{1+a}-1=\sum_{n\ge 1}(-1)^na_n.
\end{equation}
\subsection{The pre-Lie Poincar\'e-Birkhoff-Witt theorem}\label{sect:PBW}
This paragraph exposes a result by D. Guin and J-M. Oudom \cite{GO}
\begin{thm}\label{PBW}
Let $A$ be any left pre-Lie algebra, and let $S(A)$ be its symmetric algebra, i.e. the free
commutative algebra on $A$. Let $A_{\smop{Lie}}$ be the underlying Lie algebra of $A$, i.e. the vector space $A$ endowed with the Lie bracket given by $[a,b]=a\rhd b-b\rhd a$ for any $a,b\in A$, and let $\Cal U(A)$ be the enveloping algebra of  $A_{\smop{Lie}}$, endowed with its usual increasing
filtration. Let us consider the associative algebra $\Cal U(A)$ as a left module over itself.\\

There exists a left $\Cal U(A)$-module structure on $S(A)$ and a
canonical left $\Cal U(A)$-module isomorphism $\eta : \Cal U(A)\to\Cal
S(A)$, such that the associated graded  linear map $\mop{Gr}\eta:\mop{Gr}\Cal U(A)\to S(A)$ is an isomorphism of commutative graded algebras.
\end{thm}
\begin{proof}
The Lie algebra morphism 
\begin{align*}
L:A&\longrightarrow\mop{End}A\\
a&\longmapsto (L_a:b\mapsto a\rhd b)
\end{align*}
extends by Leibniz rule to a unique Lie algebra morphism $ L:A\to \mop{Der}S(A)$. Now we claim that the map $M:A\to\mop{End}S(A)$ defined by:
\begin{equation}
M_au=au+L_au
\end{equation}
is a Lie algebra morphism. Indeed we have for any $a,b\in A$ and $u\in S(A)$:
\begin{eqnarray*}
M_aM_bu&=&M_a(bu+L_bu)\\
&=&abu+aL_bu+L_a(bu)+L_aL_bu\\
&=&abu+aL_bu+bL_au+(a\rhd b)u+L_aL_bu.
\end{eqnarray*}
Hence
\begin{eqnarray*}
[M_a,M_b]u&=&(a\rhd b-b\rhd a)u+[L_a,L_b]u\\
&=&M_{[a,b]}u,
\end{eqnarray*}
which proves the claim. Now $M$ extends, by universal property of the enveloping algebra, to a unique algebra morphism $M:\Cal U(A)\to\mop{End} S(A)$. The linear map:
\begin{eqnarray*}
\eta:\Cal U(A)&\longrightarrow& S(A)\\
u&\longmapsto & M_u.1
\end{eqnarray*}
is clearly a morphism of left $\Cal U(A)$-modules. It is immediately seen by induction that for any $a_1,\ldots,a_n\in A$ we have $\eta(a_1\cdots a_n)=a_1\cdots a_n+v$ where $v$ is a sum of terms of degree $\le n-1$. This proves the theorem.
\end{proof}
\begin{rmk}
Let us recall that the symmetrization map $\sigma:\Cal U(A)\rightarrow S(A)$, uniquely determined by $\sigma(a^n)=a^n$ for any $a\in A$ and any integer $n$, is an isomorphism for the two $A_{\smop{Lie}}$-module structures given by the adjoint action. This is {\sl not\/} the case for the map $\eta$ defined above. The fact that it is possible to replace the adjoint action of $\Cal U(A)$ on itself by the simple left mutiplication is a remarkable property of pre-Lie algebras, and makes Theorem \ref{PBW} different from the usual Lie algebra PBW theorem.
\end{rmk}
Let us finally notice that, if $p$ stands for the projection from $S(A)$ onto $A$, we easily get for any $a_1,\ldots,a_k\in A$:
\begin{equation}\label{projection}
p\circ\eta(a_1\cdots a_k)=L_{a_1}\cdots L_{a_k}\un=a_1\rhd\Big(a_2\rhd\big(\cdots(a_{k-1}\rhd a_k)...\big)\Big)
\end{equation}
by a simple induction on $k$. The linear isomorphism $\eta$ transfers the
product of the enveloping algebra $\Cal U(A)$ into a noncommutative
product $\ast$ on $\Cal S(A)$ defined by:
\begin{equation}
s\ast t=\eta\big(\eta^{-1}(s)\eta^{-1}(t)\big).
\end{equation}
Suppose now that $A$ is endowed with a complete decreasing compatible filtration as in
Paragraph \ref{sect:formal-flows}. This filtration induces a complete decreasing
filtration $S(A)=S(A)_0\supset S(A)_1\supset S(A)_2\supset\cdots$, and the product $\ast$ readily extends to the completion $\widehat S(A)$. For any $a\in A$, the application of \eqref{projection} gives:
\begin{equation}
p(e^{\ast a})=W(a)
\end{equation}
as an equality in the completed symmetric algebra $\widehat S(A)$.\\

According to \eqref{diese} we can identify the pro-unipotent group
$\{e^{*a},\,a\in A\}\subset \widehat S(A)$ and the group of formal flows
of the pre-Lie algebra $A$ by means of the projection $p$, namely:
\begin{equation}\label{projectionp}
p(e^{\ast a})\# p(e^{\ast b})=p(e^{\ast a}\ast e^{\ast b})
\end{equation}
for any $a,b\in A$.
%%%%%
\section{Algebraic operads}\label{operads}
%%%%%
An operad is a combinatorial device which appeared in algebraic topology
(J-P. May, \cite{May}), coined for coding ``types of
algebras''. Hence, for example, a Lie algebra is an algebra over some operad
denoted by {\sc Lie}, an associative algebra is an algebra over some operad
denoted by {\sc Assoc}, a commutative algebra is an algebra over some operad
denoted by {\sc Com}, etc. 
%%%
\subsection{Manipulating algebraic operations}
%%%
Algebra starts in most cases with some set $E$ and some binary operation $*:E\times E\to E$. The set $E$ shows most of the time some extra structure. We will stick here to the linear setting, where $E$ is replaced by a vector space $V$ (over some base field $k$), and $*$ is bilinear, i.e. is a linear map from $V\otimes V$ into $V$. A second bilinear map is deduced from the first by permuting the entries:
\begin{equation}
a*^{\smop{op}}b:=b*a.
\end{equation}
It makes also sense to look at tri-, quadri- and multilinear operations, i.e. linear maps from $V^{\otimes n}$ to $V$ for any $V$. For example it is very easy to produce {\sl twelve\/} trilinear maps starting with the bilinear map $*$ by considering:
\begin{eqnarray*}
(a,b,c)&\mapsto& (a*b)*c,\\
(a,b,c)&\mapsto& a*(b*c),
\end{eqnarray*}
and the others deduced by permuting the three entries $a$, $b$ and $c$. One could also introduce some tri- or multilinear operations from scratch, i.e. without deriving them from the bilinear operation $*$. One can even consider $1$-ary and $0$-ary operations, the latter being just distinguished elements of $V$. Note that there is a canonical $1$-ary operation, namely the identity map $e:V\to V$. Note at this stage that the symmetric group $S_n$ obviously acts on the $n$-ary operations {\sl on the right\/} by permuting the entries before composing them.\\

The bilinear operation $*$ is not arbitrary in general: its properties determine the "type of algebra" considered. For example $V$ will be an associative or a Lie algebra if for any $a,b,c\in V$ we have respectively:
\begin{eqnarray}
(a*b)*c&=&a*(b*c),\label{eq:assoc}\\
(a*b)*c+(b*c)*a+(c*a)*b&=&0,\hskip 6mm a*b-b*a=0\label{eq:lie}.
\end{eqnarray}
The concept of operad emerges when one tries to rewrite such relations in terms of the operation $*$ only, discarding the entries $a,b,c$. For example, the associativity axiom \eqref{eq:assoc} informally expresses itself as follows: {\sl composing twice the operation $*$ in two different ways gives the same result\/}. Otherwise said:
\begin{equation}\label{eq:assoc2}
*\circ_1 *=*\circ_2 *.
\end{equation}
The Lie algebra axioms \eqref{eq:lie}, involving flip and circular permutations, clearly rewrites as:
\begin{equation}\label{eq:lie2}
(*\circ_1*)+(*\circ_1*)\circ\sigma+(*\circ_1*)\circ\sigma^2=0,\hskip 6mm *+*\circ\tau=0,
\end{equation}
where $\tau$ is the flip $(21)$ and $\sigma$ is the circular permutation $(231)$. The next paragraph will give a precise meaning to these "partial compositions", and we will end up with giving the axioms of an operad, which is the natural framework in which equations like \eqref{eq:assoc2} and \eqref{eq:lie2} make sense.
%%%
\subsection{The operad of multilinear operations}
%%%
Let us now look at the prototype of algebraic operads: for any vector space $V$, the operad $\mop{Endop}(V)$ is given
by:
\begin{equation}
\mop{Endop}(V)_n=\Cal L(V^{\otimes n},V).
\end{equation}
The right action of the symmetric group
$S_n$ on $\mop{Endop}(V)_n$ is induced by the left action of $S_n$ on
$V^{\otimes n}$ given by:
\begin{equation}
\sigma.(v_1\otimes\cdots\otimes v_n):=v_{\sigma^{-1}(1)}\otimes\cdots\otimes v_{\sigma^{-1}(n)}.
\end{equation}
Elements of  $\mop{Endop}(V)_n$ are conveniently represented as boxes with $n$ inputs and one output: as illustrated by the graphical representation below, the partial composition $a\circ_i b$ is given by:
\begin{equation}
a\circ_i b(v_1\otimes\cdots\otimes v_{k+l-1}):=a\big(v_1\otimes\cdots\otimes v_{i-1}
\otimes b(v_i\otimes\cdots\otimes v_{i+l-1})\otimes v_{i+l}\otimes\cdots\otimes v_{k+l-1}\big).
\end{equation}
\begin{equation*}
\operade
\end{equation*}
The following result is straightforward:
\begin{prop}\label{endop}
For any $a\in \mop{Endop}(V)_k$, $b\in\mop{Endop}(V)_l$, $c\in\mop{Endop}(V)_m$ one has:
\begin{align}
(a\circ_i b)\circ_{i+j-1} c&=a\circ_i(b\circ_j c),
&i\in\{1,\ldots,k\}, \ j\in \{1,\ldots,l\} \hbox{ \rm (nested associativity property)},\\
(a\circ_i b)\circ_{l+j-1}c&=(a\circ_j c)\circ_i b, &i,j\in\{1,\ldots,k\}, \ i<j \hbox{ \rm (disjoint associativity property)}.
\end{align}
The identity $e:V\to V$ satisfies the following unit property:
\begin{align}
e\circ a&=a\\
a\circ_ie&=a,\hskip 8mm i=1,\ldots ,k,
\end{align}
and finally the following equivariance property is satisfied:
\begin{equation}\label{equiv}
a.\sigma\circ_{i} b.\tau=(a\circ_{\sigma_i} b).\iota_i(\sigma,\tau)
\end{equation}
where $\iota_i(\sigma,\tau)\in S_{k+l-1}$ is defined by letting $\tau$ permute
the set $E_i=\{i,i+1,\ldots,i+l-1\}$ of cardinality $l$, and then by letting
$\sigma$ permute the set
$\{1,\ldots,i-1,E_i,i+l,\ldots,k+l-1\}$ of cardinality $k$.
\end{prop}
The two associativity properties are graphically represented as follows:
\begin{equation*}
\asoperade
\end{equation*}
%%%
\subsection{A definition for linear operads}
%%%
We are now ready to give the precise definition of an algebraic operad:
\begin{defn}
An operad $\Cal P$ (in the symmetric monoidal category
of $k$-vector spaces) is given by a collection of vector spaces
$(\Cal P_n)_{n\ge 0}$, a right action of the symmetric group
$S_n$ on $\Cal P_n$, a distinguished element $e\in\Cal P_1$, and a collection of {\sl partial compositions\/}:
\begin{align*}
\circ_i:\Cal P_k\otimes\Cal P_l &\longrightarrow \Cal P_{k+l-1},\hskip 8mm
i=1,\ldots ,k\\
(a,b)&\longmapsto a\circ_i b
\end{align*}
subject to the associativity, unit and equivariance axioms of Proposition \ref{endop}.
\end{defn}
The {\sl global
composition\/} is defined by:
\begin{align*}
\gamma:{\Cal P}_n\otimes{\Cal P}_{k_1}\otimes\cdots\otimes{\Cal P}_{k_n}&\longrightarrow
{\Cal P}_{k_1+\cdots +k_n}\\
(a,b_1,\ldots,b_n)&\longmapsto \Big(...\big((a\circ_n
b_n)\circ_{n-1}b_{n-1}\big)\cdots\Big)\circ_1 b_1
\end{align*}
and is graphically represented as follows:
\begin{equation*}
\gcoperade
\end{equation*}
The operad ${\Cal P}$ is \textsl{augmented} if ${\Cal P}_0=\{0\}$ and ${\Cal
  P}_1=k.e$. For any
operad $\Cal P$, a $\Cal P$-algebra structure on the vector space $V$ is a
morphism of operads from $\Cal P$ to $\mop{Endop}(V)$. For any two $\Cal
P$-algebras $V$ and $W$, a morphism of $\Cal P$-algebras is a linear map
$f:V\to W$ such that for any $n\ge 0$ and for any $\gamma\in\Cal P_n$ the
following diagram commutes,
\diagramme{
\xymatrix{V^{\otimes n}
\ar[rr]^{f^{\otimes n}}\ar[d]^{\gamma}	&&W^{\otimes n}\ar[d]^{\gamma}\\
V	\ar[rr]_f
&&W
}
}
\noindent where we have denoted by the same letter $\gamma$ the element of $\Cal P_n$
and its images in $\mop{Endop}(V)_n$ and $\mop{Endop}(W)_n$.\\

Now let $V$ be any $k$-vector space. The \textsl{free
  $\Cal P$-algebra} is a $\Cal P$-algebra $\Cal F_{\Cal P}(V)$ endowed with a
linear map $\iota:V\inj{6}\Cal F_{\Cal P}(V)$ such that for any $\Cal
P$-algebra $A$ and for any linear map $f:V\to A$ there is a unique $\Cal
P$-algebra morphism $\overline f:\Cal F_{\Cal P}(V)\to A$ such that
$f=\iota\circ\overline f$. The free $\Cal P$-algebra $\Cal F_{\Cal P}(V)$ is
unique up to isomorphism, and one can prove that a concrete presentation of it
is
given by:
\begin{equation}\label{free-alg}
\Cal F_{\Cal P}(V)=\bigoplus_{n\ge 0}\Cal P_n\otimes_{S_n}V^{\otimes n},
\end{equation}
the map $\iota$ being obviously defined. When $V$ is of finite dimension $d$,
the corresponding free $\Cal P$-algebra is often called the \textsl{free $\Cal
P$-algebra with $d$ generators}.\\

There are several other equivalent definitions for an operad. For more details about operads, see e.g. \cite{Loday96}, \cite{LV}.
%%%
\subsection{A few examples of operads}
%%%
\subsubsection{The operad {\sc Assoc}}
This operad governs associative algebras. $\hbox{\sc{Assoc}}_n$ is given by $k[S_n]$ (the algebra of the symmetric group $S_n$) for any $n\ge 0$, whereas $\hbox{\sc{Assoc}}_0:=\{0\}$. The right action of $S_n$ on $\hbox{\sc{Assoc}}_n$ is given by linear extension of right multiplication:
\begin{equation}
\left(\sum_i\lambda_i\sigma_i\right)\sigma:=\sum_i\lambda_i(\sigma_i\sigma).
\end{equation}
Let $\sigma\in\hbox{\sc{Assoc}}_k$ and $\tau\in\hbox{\sc{Assoc}}_l$. The partial compositions are given for any $i=1,\ldots,k$ by:
\begin{equation}
\sigma\circ_i\tau:=\iota_i(\sigma,\tau),
\end{equation}
with the notations of \eqref{equiv}. The reader is invited to check the two associativity axioms, as well as the equivariance axiom which reads:
\begin{equation}
(\sigma\sigma')\circ_i(\tau\tau')=(\sigma\circ_{\sigma'(i)}\tau)(\sigma'\circ_i\tau')
\end{equation}
for any $\sigma,\sigma'\in\hbox{\sc{Assoc}}_k$ and $\tau,\tau'\in\hbox{\sc{Assoc}}_l$. Let us denote by $e_k$ the unit element in the symmetric group $S_k$. We obviously have $e_k\circ_i e_l=e_{k+l-1}$ for any $i=1,\ldots ,k$. In particular,
\begin{equation}\label{assoc1}
e_2\circ_1 e_2=e_2\circ_2 e_2=e_3.
\end{equation}
Now let $V$ be an algebra over the operad $\hbox{\sc{Assoc}}$, and let $\Phi:\hbox{\sc{Assoc}}\to\mop{Endop}(V)$ be the corresponding morphism of operads. Let $\mu:V\otimes V\to V$ be the binary operation $\Phi(e_2)$. In view of \eqref{assoc1} we have:
\begin{equation}\label{assoc2}
\mu\circ_1\mu=\mu\circ_2\mu.
\end{equation}
In other words, $\mu$ is associative. As $e_k$ can be obtained, for any $k\ge 3$, by iteratively composing $k-2$ times the element $e_2$, we see that any element of $\hbox{\sc{Assoc}}_k$ can be obtained from $e_2$, partial compositions, symmetric group actions and linear combinations. As a consequence, any $k$-ary operation on $V$ which is in the image of $\Phi$ can be obtained in terms of the associative product $\mu$, partial compositions, symmetric group actions and linear combinations. Summing up, an algebra over the operad $\hbox{\sc{Assoc}}$ is nothing but an associative algebra. In view of \eqref{free-alg}, the free $\hbox{\sc{Assoc}}$-algebra over a vector space $W$ is the (non-unital) tensor algebra $T^+(W)=\bigoplus_{k\ge 1}W^{\otimes k}$.\\

In the same line of thoughts, the operad governing unital associative algebras is defined similarly, except that the space of $0$-ary operations is $k.e_0$ with $e_k\circ_i e_0=e_{k-1}$ for any $i=1,\ldots ,k$. The unit element $u:k\to V$ of the algebra $V$ is given by $u=\Phi(e_0)$. The free unital algebra over a vector space $W$ is the full tensor algebra $T(W)=\bigoplus_{k\ge 0}W^{\otimes k}$.
\subsubsection{The operad $\hbox{\sc Com}$}
This operad governs commutative associative algebras. $\hbox{\sc{Com}}_n$ is one-dimensional for any $n\ge 1$, given by $k.\overline e_n$ for any $n\ge 0$, whereas $\hbox{\sc{Com}}_0:=\{0\}$. The right action of $S_n$ on $\hbox{\sc{Com}}_n$ is trivial. The partial compositions are defined by:
\begin{equation}
\overline e_k\circ_i \overline e_l=\overline e_{k+l-1}\hbox{ for any }i=1,\ldots ,k.
\end{equation}
The three axioms of an operad are obviously verified. Let $V$ be an algebra over the operad $\hbox{\sc Com}$, and let $\Phi:\hbox{\sc{Com}}\to\mop{Endop}(V)$ be the corresponding morphism of operads. Let $\mu:V\otimes V\to V$ be the binary operation $\Phi(\overline e_2)$. We obviously have:
\begin{equation}
\mu\circ_1\mu=\mu\circ_2\mu,\hskip 6mm \mu=\mu.\tau,
\end{equation}
where $\tau\in S_2$ is the flip. Hence $\mu$ is associative and commutative. Here, any $k$-ary operation in the image of $\Phi$ can be obtained, up to a scalar, by iteratively composing $\overline e_2$ with itself. Hence an algebra over the operad $\hbox{\sc{Com}}$ is nothing but a commutative associative algebra. In view of \eqref{free-alg}, the free $\hbox{\sc{Com}}$-algebra over a vector space $W$ is the (non-unital) symmetric algebra $S^+(W)=\bigoplus_{k\ge 1}S^k(W)$.\\

the operad governing unital commutative associative algebras is defined similarly, except that the space of $0$-ary operations is $k.\overline e_0$ with $\overline e_k\circ_i \overline e_0=\overline e_{k-1}$ for any $i=1,\ldots ,k$. The unit element $u:k\to V$ of the algebra $V$ is given by $u=\Phi(e_0)$. The free unital algebra over a vector space $W$ is the full symmetric algebra $S(W)=\bigoplus_{k\ge 0}S^k(W)$.\\

The map $e_k\to\overline e_k$ is easily seen to define a morphism of operad $\Psi:\hbox{\sc Assoc}\to\hbox{\sc Com}$. Hence any $\hbox{\sc Com}$-algebra is also an $\hbox{\sc Assoc}$-algebra. This expressed the fact that, forgetting commutativity, a commutative associative algebra is also an associative algebra.
\subsubsection{Associative algebras}
Any associative algebra $A$ is some degenerate form of operad: indeed, defining $\Cal P(A)$ by $\Cal P(A)_1:=A$ and $\Cal P(A)_n:=\{0\}$ for $n\not =1$, the collection $\Cal P(A)$ is obviously an operad. An algebra over $\Cal P(A)$ is the same as an $A$-module.\\

This point of view leads to a more conceptual definition of operads: an operad is nothing but an associative unital algebra in the category of "S-objects", i.e. collections of vector spaces $(\Cal P_n)_{n\ge 0}$ with a right action of $S_n$ on $\Cal P_n$. There is a suitable "tensor product" $\boxtimes$ on S-objects, however not symmetric, such that the global composition $\gamma$ and the unit $u:k\to\Cal P_1$ (defined by $u(1)=e$) make the following diagrams commute:
\diagramme{
\xymatrix{
\Cal P\boxtimes \Cal P\boxtimes A \ar[d]^{I\boxtimes \gamma}\ar[r]^{\gamma\boxtimes I}
	&\Cal P\boxtimes \Cal P \ar[d]^\gamma	\\
\Cal P\boxtimes \Cal P \ar[r]^\gamma	& \Cal P,}
}
\diagramme{
\xymatrix{
k\boxtimes \Cal P \ar[r]^{u\boxtimes I} \ar[dr]^\sim	& \Cal P\boxtimes \Cal P \
 \ar[d]_\gamma	&\Cal P\boxtimes k \ar[l]_{I\boxtimes u}\ar[dl]_\sim	\\
&\Cal P.&}
.}
These two diagrams commute if and only if $e$ verifies the unit axiom and the partial compositions verify the two associativity axioms and the equivariance axiom \cite{LV}.

%%%%%
\section{Pre-Lie algebras (continued)}
%%%%%
%%%%%
\subsection{Pre-Lie algebras and augmented operads}
%%%%%
\subsubsection{General construction}
%%%%%
We adopt the notations of Paragraph \ref{operads}. The sum of the partial compositions yields a right pre-Lie algebra structure
on the free $\Cal P$-algebra with one generator, more precisely on $F_{\Cal P}^+:=\bigoplus_{n\ge 2}{\Cal P}_n/S_n$, namely:
\begin{equation}
\overline a\lhd\overline b:=\sum_{i=1}^k \overline{a\circ_i b}.
\end{equation}
Following F. Chapoton \cite{C} one can consider the pro-unipotent
group $G_{\Cal P}^e$ associated with the completion of the pre-Lie algebra
$F_{\Cal P}^+$ for the filtration induced by the grading. More precisely
Chapoton's group $G_{\Cal P}$ is given by the elements $g\in\widehat{F_{\Cal
    P}}$ such that $g_1\not =0$, whereas $G_{\Cal P}^e$ is the subgroup of
$G_{\Cal P}$ formed by elements $g$ such that $g_1=e$.\\

Any element $a\in{\Cal P}_n$ gives rise to an $n$-ary
operation $\omega_a:F_{\Cal P}^{\otimes n}\to F_{\Cal P}$, and for any $x,y_1,\ldots ,y_n\in
F_{\Cal P}^+$ we have\footnote{We thank Muriel Livernet for having brought
  this point to our attention.}\cite{MS}:
\begin{equation}\label{derivation}
\omega_a(y_1,\ldots ,y_n)\lhd x=\sum_{j=1}^n \omega_a(y_1,\ldots,y_j\lhd x,\ldots,y_n).
\end{equation}
%%%%%
\subsubsection{The pre-Lie operad}\label{sect:plo}
%%%%%
Pre-Lie
algebras are algebras over the {\sl pre-Lie operad\/}, which has been
described in detail by F. Chapoton and M. Livernet in \cite{ChaLiv} as
follows: $\Cal P\Cal L_n$ is the vector space of labelled rooted trees, and
partial composition $s\circ_i t$ is given by summing all the possible ways of
inserting the tree $t$ inside the tree $s$ at the vertex labelled by $i$.\\

The free left pre-Lie algebra with one generator is then given by the space
$\Cal T=\bigoplus_{n\ge 1}T_n$ of
rooted trees, as quotienting with the symmetric group actions amounts to
neglect the labels. The pre-Lie operation $(s,t)\mapsto (s\rightarrow t)$ is given by
the sum of the graftings of $s$ on $t$ at all vertices of $t$. As a
consequence of \eqref{derivation} we have two pre-Lie operations on $\Cal
T'=\bigoplus_{n\ge 2}T_n$ which interact as follows \cite{MS}:
\begin{equation}\label{derivation2}
(s\rightarrow t)\lhd u=(s\lhd u)\rightarrow t+s\rightarrow(t\lhd u).
\end{equation}
The first pre-Lie operation $\lhd$ comes from the fact that $\Cal P\Cal L$ is
an augmented operad, whereas the second pre-Lie operation $\rightarrow$ comes from
the fact that $\Cal P\Cal L$ is the pre-Lie operad itself! Similarly:
\begin{thm}\label{fpl}
The free
pre-Lie algebra with $d$ generators is the vector space of rooted trees with
$d$ colours, endowed with grafting.
\end{thm}
%%%%%
\subsection{A pedestrian approach to free pre-Lie algebra}\label{sect:fpl}
%%%%%
We give in this paragraph a direct proof of Theorem \ref{fpl} without using
operads. It is similar to the proof of the main theorem in \cite{ChaLiv} about the structure of the pre-Lie operad, except that we consider unlabelled trees. We stick to $d=1$ (i.e. one generator), the proof for several
generators beeing completely analogous. Let $\Cal T$ be the vector space
spanned by rooted trees. First of all, the grafting operation
is pre-Lie, because for any trees $s$, $t$ and $u$ in $\Cal T$ the expression:
\begin{equation}
s\to(t\to u)-(s\to t)\to u
\end{equation}
is obtained by summing up all the possibilities of grafting $s$ and $t$ at
some vertex of $u$. As such it is obviously symmetric in $s$ and $t$. Now let
$(A,\rhd)$ be any left pre-lie algebra, and choose any $a\in A$. In order to
prove Theorem \ref{fpl} for one generator, we have to show that there is a
unique pre-Lie algebra morphism $F_a:\Cal T\longrightarrow A$ such that
$F_a(\racine)=a$. We obtain easily for the first trees:
\begin{eqnarray*}
F_a(\racine)&=&a\\
F_a(\arbrea)&=&a\rhd a\\
F_a(\arbreba)&=&(a\rhd a)\rhd a\\
F_a(\arbrebb)&=&a\rhd(a\rhd a)-(a\rhd a)\rhd a.
\end{eqnarray*}
Can we continue like this? We proceed by double induction, firstly on the
number of vertices, secondly on the number of branches, i.e. the valence of
the root. Write any tree $t$ with $n$ vertices as $t=B_+(t_1,\ldots,t_k)$,
where the $t_j$'s are the branches, and where $B_+$ is the operator which
grafts the branches on a common extra root. By the induction hypothesis on
$n$, the images $F_a(t_j)$ are well-defined.\\

Suppose first that $k=1$, i.e. $t=B_+(t_1)=t_1\to\racine$. Then we obviously
have $F_a(t)=F_a(t_1)\rhd a$. Suppose now that $F_a(s)$ is unambiguously
defined for any tree $s$ with $n$ vertices and $k'$ branches with $k'\le
k-1$. The equation:
\begin{eqnarray*}
t&=&B_+(t_1,\ldots,t_k)\\
&=&t_1\to B_+(t_2,\ldots,t_k)-\sum_{j=2}^kB_+(t_2,\ldots,t_{j-1},t_1\to t_j,t_{j+1},\ldots,t_n)
\end{eqnarray*}
shows that, if $F_a(t)$ exists, it is uniquely defined by:
\begin{equation}\label{def-fa}
F_a(t)=F_a(t_1)\rhd F_a\big(B_+(t_2,\ldots,t_k)\big)-\sum_{j=2}^kF_a\big(B_+(t_2,\ldots,t_{j-1},t_1\to t_j,t_{j+1},\ldots,t_n)\big).
\end{equation}
What remains to be shown is that this expression does not depend on the choice
of the distinguished branch $t_1$. In order to see this, choose a second
branch (say $t_2$), and consider the expression:
\begin{equation}
T:=t_1\to\big(t_2\to B_+(t_3,\ldots,t_k)\big)-(t_1\to t_2)\to B_+(t_3,\ldots,t_k),
\end{equation}
which is obtained by grafting $t_1$ and $t_2$ on $B_+(t_3,\ldots,t_k)$. This
expression is the sum of five terms:
\begin{enumerate}
\item $T_1$, obtained by grafting $t_1$ and $t_2$ on the root. It is nothing
  but the tree $t$ itself.
\item $T_2$, obtained by grafting $t_1$ on the root and $t_2$ elsewhere.
\item $T_3$, obtained by grafting $t_2$ on the root and $t_1$ elsewhere.
\item $T_4$, obtained by grafting $t_1$ on some branch and $t_2$ on some other
  branch.
\item $T_5$, obtained by grafting $t_1$ and $t_2$ on the same branch.
\end{enumerate}
The terms $F_a(T_2)+F_a(T_3)$, $F_a(T_4)$ and $F_a(T_5)$ are well-defined by
the induction hypothesis on the number of branches, and obviously symmetric in
$t_1$ and $t_2$. We thus arrive at:
\begin{eqnarray*}
F_a(t)&=&F_a(t_1)\rhd\Big(F_a(t_2)\rhd F_a\big(
B_+(t_3,\ldots,t_k)\big)\Big)-\big(F_a(t_1)\rhd F_a(t_2)\big)\rhd F_a\big(
B_+(t_3,\ldots,t_k)\big)\\
&-&F_a(T_2)-F_a(T_3)-F_a(T_4)-F_a(T_5),
\end{eqnarray*}
which is symmetric in $t_1$ and $t_2$ thanks to the left pre-Lie relation in
$A$. The expression \eqref{def-fa} is then the same if we exchange $t_1$ with
the branch $t_2$ or any other branch, hence it is invariant by any permutation
of the branches $t_1,\ldots ,t_n$. This proves Theorem \ref{fpl} for one generator. The general case is
proven similarly except that we have to replace $a\in A$ by a collection
$\{a_1,\ldots ,a_d\}$.
%%%%% 
\subsection{Right-sided commutative Hopf algebras and the Loday-Ronco theorem}\label{sect:lodayronco}
%%%%%
J.-L. Loday and M. Ronco have found a deep link between pre-Lie algebras and
commutative Hopf algebras of a certain type: let $\Cal H$ be a commutative Hopf algebra. Following \cite{LR08}, we say that $\Cal H$ is {\sl right-sided\/} if it is free as a commutative algebra, i.e. $\Cal H=S(V)$ for some $k$-vector space $V$, and if the reduced coproduct verifies~:
\begin{equation}
\wt\Delta(V)\subset\Cal H\otimes V.
\end{equation} 
Suppose moreover that $V=\bigoplus_{n\ge 0}$ is graded with finite-dimensional
homogeneous components. Then the graded dual $A=V^0$ is a left pre-Lie
algebra, and by the Milnor-Moore theorem, the graded dual $\Cal H^0$ is
isomorphic to the enveloping algebra $\Cal U(A_{\smop{Lie}})$ as graded Hopf
algebra. Conversely, for any graded pre-Lie algebra $A$ the graded dual $\Cal
U(A_{\smop{Lie}})^0$ is free commutative right-sided (\cite{LR08} Theorem
5.3).\\

The Hopf algebra $\Cal H_{CK}$ of rooted forests enters into this framework,
and, as it was first explicited in \cite{C01}, the associated pre-Lie algebra is the free pre-Lie algebra of rooted trees
with grafting: to see this, denote by $(\delta_s)$ the dual basis in the graded dual $\Cal H_{CK}^\circ$ of the forest basis of $\Cal H_{CK}$. The correspondence $\delta:s\mapsto \delta_s$ extends linearly to a unique vector space isomorphism from $\Cal H_{CK}$ onto $\Cal H_{CK}^\circ$. For any tree $t$ the corresponding $\delta_t$ is an infinitesimal character of $\Cal H_{CK}$, i.e. it is a primitive element of $\Cal H^\circ$. We denote by $\ast$ the (convolution) product of $\Cal H^\circ$. We have~:
\begin{equation}
\delta_t\ast \delta_u-\delta_u\ast \delta_t=\delta_{t\curvearrowright u-u\curvearrowright t}.
\end{equation}
Here $t\curvearrowright u$ is obtained by grafting $t$ on $u$, namely:
\begin{equation}
t\curvearrowright u=\sum_{v}N'(t,u,v)v,
\end{equation}
where $N'(t,u,v)$ is the number of partitions $V(t)=V\amalg W, W<V$ such that $v\restr{V}=t$ and $v\restr{W}=u$.
Another normalization is often employed: considering the normalized dual basis
$\wt\delta_t=\sigma(t)\delta_t$, where $\sigma(t)=|\mop{Aut}t|$ stands for the
symmetry factor of $t$, we obviously have:
\begin{equation}
\wt\delta_t\ast \wt\delta_u-\wt\delta_u\ast \wt\delta_t=\wt\delta_{t\to u-u\to t},
\end{equation}
where:
\begin{equation}
t\to u=\sum_{v}M'(t,u,v)v,
\end{equation}
where $M'(t,u,v)=\displaystyle\frac{\sigma(t)\sigma(u)}{\sigma(v)}N'(t,u,v)$
can be interpreted as the number of ways to graft the tree $t$ on the tree $u$
in order to get the tree $v$. The operation $\to$ then coincides with the
grafting free pre-Lie operation introduced in Paragraph
\ref{sect:plo}\footnote{The two notations $\to$ and $\curvearrowright$ come from
  \cite{CEM}, but have been exchanged.}.\\

The other pre-Lie operation $\lhd$ of Paragraph \ref{sect:plo}, more precisely
its opposite $\rhd$, is associated to another right-sided Hopf algebra of
forests $\Cal H$ which has been investigated in \cite{CEM} and \cite{MS}, and
which can be defined by considering trees as Feynman diagrams (without loops):
Let $\Cal T'$ be the vector space spanned by
rooted trees with at least one edge. Consider the symmetric algebra $\Cal H =
\Cal S(\Cal T')$, which can be seen as the $k$-vector space generated by rooted
forests with all connected components containing at least one edge. One
identifies the unit of $ \Cal S(\Cal T')$ with the rooted tree $\racine$. A {\sl subforest\/} of a tree $t$ is either the trivial forest $\racine$, or a collection $(t_1,\ldots ,t_n)$ of pairwise disjoint subtrees of $t$, each of them containing at least one edge. In particular two subtrees of a subforest cannot have any common vertex.\\

Let $s$ be a subforest of a rooted tree $t$. Denote by $t / s$ the tree obtained by contracting each connected component of $s$ onto a vertex. We turn $\Cal H $ into a bialgebra by defining a coproduct $\Delta : \Cal H \to \Cal H \otimes  \Cal H$ on each tree $t \in \Cal T'$ by~:
\begin{equation}
\label{newcoprod}
	\Delta(t)=\sum_{s\subseteq t} s \otimes t/s,
\end{equation}
where the sum runs over all possible subforests (including the unit $\racine$
and the full subforest $t$). As usual we extend the coproduct $\Delta$
multiplicatively onto $ \Cal S(\Cal T')$. In fact, co-associativity is easily verified. This makes $\Cal H:=\bigoplus_{n \ge 0}\Cal H_n$ a connected graded bialgebra, hence a Hopf algebra, where the grading is defined in terms of the number of edges. The antipode $S : \Cal H \to \Cal H$ is given (recursively with respect to the number of edges) by one of the two following formulas:
\begin{eqnarray}
	S(t) & = & -t-\sum_{s, \racine \, \not= s\subsetneq t}S(s)\ t/s, \\
	S(t) & = & -t-\sum_{s, \racine \, \not= s\subsetneq t}s\ S(t/s).
	\label{eq-anti-recu}
\end{eqnarray}
It turns out that $\Cal H_{CK}$ is left comodule coalgebra over $\Cal H$
\cite{CEM}, \cite{MS}, in the sense that the following diagram commutes:
\diagramme{
\xymatrix{{\mathcal H}_{CK}\ar[r]^{\Phi}\ar[d]_{\Delta_{CK}}
&{\mathcal H}\otimes{\mathcal H}_{CK}\ar[dd]^{I\otimes\Delta_{CK}}\\
{\mathcal H}_{CK}\otimes{\mathcal H}_{CK}\ar[d]_{\Phi\otimes\Phi}&\\
{\mathcal H}\otimes{\mathcal H}_{CK}\otimes{\mathcal H}\otimes{\mathcal
  H}_{CK}
\ar[r]_{m^{1,3}}&{\mathcal H}\otimes{\mathcal H}_{CK}\otimes{\mathcal H}_{CK}
}
}
Here the coaction $\Phi:\Cal H_{CK}\to \Cal H\otimes\Cal H_{CK}$ is
the algebra morphism 
given by $\Phi(\un)=\bullet\otimes \un$ and $\Phi(t)=\Delta_{\Cal H}(t)$ for
any nonempty tree $t$. As a consequence, the group of characters of $\Cal H$
acts on the group of characters of $\Cal H_{CK}$ by automorphisms.
%%%%%%%%%%%%%%%%%%%%%%%%%%%%%%%%%%%%
\subsection{Pre-Lie algebras of vector fields}\label{sect:vf}
%%%%%%%%%%%%%%%%%%%%%%%%%%%%%%%%%%%%
\subsubsection{Flat torsion-free connections}
Let $M$ be a differentiable manifold, and let $\nabla$ the covariant
derivation operator associated to a connection on the tangent bundle
$TM$. The covariant derivation is a bilinear operator on vector fields (i.e. two sections of the tangent
bundle): $(X,Y)\mapsto \nabla_XY$ such that the following axioms are fulfilled:
\begin{align*}
\nabla_{fX}Y&=f\nabla_XY,\\
\nabla_X(fY)&=f\nabla_X Y+(X.f)Y\hbox{ (Leibniz rule)}.
\end{align*}
The torsion of the connection $\tau$ is defined by:
\begin{equation}
\tau(X,Y)=\nabla_XY-\nabla_YX-[X,Y],
\end{equation}
and the curvature tensor is defined by:
\begin{equation}
R(X,Y)=[\nabla_X,\nabla_Y]-\nabla_{[X,Y]}.
\end{equation}
The connection is {\sl flat\/} if the curvature $R$ vanishes identically, and
{\sl torsion-free\/} if $\tau=0$. The following crucial observation by Y. Matsushima, as early as 1968 \cite[Lemma 1]{M68} is an
immediate consequence of \eqref{prelie1-2}:
\begin{prop}
For any smooth manifold $M$ endowed with a flat torsion-free connection
$\nabla$, the space $\chi(M)$ of vector fields is a left pre-Lie algebra, with
pre-Lie product given by:
\begin{equation}
X\rhd Y:=\nabla_XY.
\end{equation}
\end{prop}
Note that on $M=\R^n$ endowed with its canonical flat torsion-free
connection, the pre-Lie product is given by:
\begin{equation}\label{coord}
(f_i\partial_i)\rhd(g_j\partial_j)=f_i(\partial_ig_j)\partial_j.
\end{equation}
%%%%%
\subsubsection{Relating two pre-Lie structures}
%%%%%
As early as 1857, A. Cayley \cite{Cay} discovered a link between rooted trees
and vector fields on the manifold $\R^n$ endowed with its natural flat torsion
free connection, which can be described in modern terms as follows: let $\T$ be the free pre-lie algebra on the space of vector
fields on $\R^n$. A basis of $\T$ is given by rooted trees with vertices
decorated by some basis of $\chi(\R^n)$. There is a unique pre-Lie algebra
morphism $\Cal Y$, the \textsl{Cayley map}, such that $\Cal Y(\bullet_X)=X$
for any vector field $X$.
\begin{prop}
For any rooted tree $t$, each vertex $v$ being decorated by a vector field
$X_v$, the vector field $\Cal Y(t)$ is given at $x\in\R^n$ by the following
recursive procedure \cite{HWL}: if the decorated tree
$t$ is obtained by grafing all its branches $t_k$ on the root $r$ decorated by the
vector field $X_r=\sum_{i=1}^n f_i\partial_i$, i.e. if it writes
$t=B_+^{X_r}(t_1,\ldots,t_k)$, then:
\begin{align}
\Cal Y(\bullet_{X_r})&=X_r,\\
\Cal Y(t)&=\sum_{i=1}^n \Cal Y_i(t)\partial_i \hbox{ with :}\nonumber\\
\Cal Y_i(t)(x)&=f_i^{(k)}(x)\big(\Cal Y(t_1)(x),\ldots,\Cal Y(t_k)(x)\big),
\end{align}
where $f_i^{(k)}(x)$ stands for the $k^{\smop{th}}$ differential of $f_i$.
\end{prop}
\begin{proof}
From \eqref{coord} we get for any vector field $X$ and any other vector field
$Y=\sum_{j=1}^n g_j\partial_j$:
\begin{equation}
X\rhd\sum_{j=1}^n g_j\partial_j=\sum_j(X.g_j)\partial_j.
\end{equation}
In other words, $X\rhd Y$ is the derivative of $Y$ along the vector field $X$,
where $Y$ is viewed as a $C^{\infty}$ map from $\R^n$ to $\R^n$. We prove the
result by induction on the number $k$ of branches: for $k=1$ we check:
\begin{eqnarray*}
\Cal Y(s\to\bullet_Y)(x)&=&\big(\Cal Y(s)\rhd\Cal Y(\bullet_Y)\big)(x)\\
&=&\big(\Cal Y(s)\rhd Y\big)(x)\\
&=&\sum_{j=1}^n\big(\Cal Y(s).g_j\big)(x)\partial_j\\
&=&\sum_{j=1}^ng'_j(x)\big(\Cal Y(s)(x)\big)\partial_j.
\end{eqnarray*}
Now we can compute, using the Leibniz rule and the induction hypothesis (we
drop the point $x\in\R^n$ where the vector fields are evaluated):
\begin{eqnarray*}
\Cal Y\big(B_+^Y(t_1,\ldots,t_k)\big)&=& \Cal Y\big(t_1\to B_+^Y(t_2,\ldots,t_k)\big)
-\sum_{r=2}^k \Cal Y\big(B_+^Y(t_2,\ldots,t_{r-1},t_1\to
t_r,t_{r+1},\ldots,t_k)\big)\\
&\hskip -80mm=& \hskip-44mm\Cal Y(t_1)\rhd\sum_{j=1}^n f_i^{(k-1)}\big(\Cal Y(t_2),\ldots,\Cal
Y(t_k)\big)\partial_j-\sum_{r=2}^k\sum_{j=1}^n f_i^{(k-1)}\big(\Cal Y(t_2),\ldots,\Cal
Y(t_{r-1}),\Cal Y(t_1)\rhd \Cal Y(t_r),\Cal
Y(t_{r-1}),\ldots,\Cal Y(t_k)\big)\partial_j\\
&\hskip -80mm=& \hskip-44mm \sum_{j=1}^nf_i^{(k)}\big(\Cal Y(t_1),\ldots,\Cal
Y(t_k)\big)\partial_j.
\end{eqnarray*}
\end{proof}
\begin{cor}[closed formula]\label{cf}
For any rooted tree $t$ with set of vertices $\Cal V(t)$ and root $r$, each vertex $v$ being decorated by a vector field
$X_v=\sum_{i=1}^nX_v^i\partial_i$, the vector field $\Cal Y(t)$ is given at $x\in\R^n$ by the following formula:
\begin{equation}
\Cal Y(t)(x)=\sum_{F:\Cal V(t)\to\{1,\ldots ,n\}}\left(\prod_{v\in\Cal V(t)-\{r\}}
\partial_{I(v)}(X_v^{F(v)})(x)\right)\partial_{I(r)}(X_r^{F(r)})(x)\partial_{F(r)},
\end{equation}
with the shorthand notation:
\begin{equation}
\partial_{I(v)}:=\prod_{w\to v}\partial_{F(w)},
\end{equation}
where the product runs over the incoming vertices of $v$.
\end{cor}
Now fix a vector field $X$
on $\R^n$ and consider the map $d_X$ from
rooted trees to vector field-decorated rooted
trees, which decorates each vertex by $X$. It is obviously a pre-Lie algebra
morphism, and $F_X:=\Cal Y\circ d_X$ is the unique pre-Lie algebra morphism
which sends the one-vertex tree $\bullet$ to the vector field $X$.
%%%%%
\subsection{B-series, composition and substitution}
%%%%%
B-series have been defined by E. Hairer and G. Wanner, following the pioneering work of J. Butcher \cite{B63} on Runge-Kutta methods for the numerical resolution of differential equations. Remarkably enough, rooted trees revealed to be an adequate tool not only for vector fields, but also for the numerical approximation of their integral curves. J. Butcher discovered that the Runge-Kutta methods form a group (since then called the Butcher group), which is nothing but the character group of the Connes-Kreimer Hopf algebra $\Cal H_{CK}$ \cite{Br}.\\

Consider any left pre-Lie algebra $(A,\rhd)$, introduce a fictitious unit
$\un$ such that $\un\rhd a=a\rhd\un=a$ for any $a\in A$, and consider
for any $a\in A$ the unique left pre-Lie algebra morphism $F_a:(
\Cal T,\to)\longrightarrow (A,\rhd)$ such that $F_a(\bullet)=a$.
A {\sl $B$-series\/} is an element of $hA[[h]]\oplus k.\un$ defined by:
\begin{equation}\label{def:B-series}
B(\alpha;a):=\alpha(\emptyset)\un+\sum_{s\in T}h^{v(s)}\frac{\alpha(s)}{\sigma(s)}F_a(s),
\end{equation}
where $\alpha$ is any linear form on $\Cal T\oplus k\emptyset$ (here $\sigma(s)$ is the symmetry factor of the tree, i.e. the order of its group of automorphisms). It matches the usual notion of $B$-series \cite{HWL} when $A$ is the pre-Lie
algebra of vector fields on $\R^n$ (it is also convenient to set 
$F_a(\emptyset)=\un$). In this case, the vector fields $F_a(t)$ for a tree $t$ are
differentiable maps from $\R^n$ to $\R^n$ called {\sl
  elementary differentials\/}. $B$-series can be composed coefficientwise, as
series in the indeterminate $h$ whose coefficients are maps from $\R^n$ to
$\R^n$. The same definition with trees decorated by a
set of colours $\Cal D$ leads to straightforward generalizations. For example
$P$-series used in partitioned Runge-Kutta methods \cite{HWL} correspond to
bi-coloured trees.\\

A slightly different way of defining $B$-series is the following: consider the
unique pre-Lie algebra morphism $\Cal F_a:\Cal T\to hA[h]$ such that $\Cal
F_a(\bullet)=ha$. It respects the gradings given by the number of vertices
and the powers of $h$ respectively, hence it extends to $\Cal F_a:\widehat
{\Cal T}\to hA[[h]]$, where $\widehat {\Cal T}$ is the completion of $\Cal T$ with respect to the grading. We further extend it to the empty tree by setting $\Cal F_a(\emptyset)=\un$. We have then:
\begin{equation}\label{Bseriesbis}
B(\alpha;a)=\Cal F_a\circ \widetilde\delta^{-1}(\alpha),
\end{equation}
where $\widetilde\delta$ is the isomorphism from $\widehat {\Cal T}\oplus k\emptyset$
to $(\Cal T\oplus k\emptyset)^*$ given by the normalized dual basis (see
Paragraph \ref{sect:lodayronco}).\\

We restrict ourselves to $B$-series $B(\alpha;a)$ with $\alpha(\emptyset)=1$. Such $\alpha$'s are in one-to-one
correspondence with
characters of the algebra of forests (which is the underlying algebra of
$\Cal H_{CK}$) by setting:
\begin{equation}
\alpha(t_1\cdots t_k):=\alpha(t_1)\cdots\alpha(t_k).
\end{equation}
The Hairer-Wanner theorem \cite[Theorem III.1.10]{HWL} says that composition
of $B$-series corresponds to the convolution product of characters of $\Cal
H_{CK}$, namely:
\begin{equation}\label{hw-theorem}
B(\beta;a)\circ B(\alpha;a)=B(\alpha * \beta,a),
\end{equation}
where linear forms $\alpha,\beta$ on $\Cal T\oplus k\emptyset$ and their character counterparts are
identified modulo the above correspondence.\\

Let us now turn to substitution, after \cite{CHV}. The idea is to replace the
vector field $a$ in a $B$-series $B(\beta;a)$ by another vector field $\widetilde a$
which expresses itself as a $B$-series, i.e. $\widetilde a=h^{-1}B(\alpha;a)$ where
$\alpha$ is now a linear form on $\Cal T\oplus k\emptyset$ such that
$\alpha(\emptyset)=0$. We suppose here moreover that $\alpha(\bullet)=1$. Such
$\alpha$'s are in one-to-one correspondence with characters of $\widetilde {\Cal
  H}$. The following proposition is proved in \cite{CEM}:
\begin{prop}\label{substitution}
For any linear forms $\alpha,\beta$ on $\Cal T$ with $\alpha(\bullet=1)$, we have:
\begin{equation}
B\big(\beta;\frac 1h B(\alpha;a)\big)=B(\alpha\star\beta;a),
\end{equation}
where $\alpha$ is multiplicatively extended to
forests, $\beta$ is seen as an infinitesimal character of $\Cal H_{CK}$ and
where $\star$ is the dualization of the left coaction $\Phi$ of $\Cal H$ on
$\Cal H_{CK}$ defined in Paragraph \ref{sect:lodayronco}.
\end{prop}
The condition $\alpha(\bullet)=1$ is in fact dropped in \cite[Proposition 15]{CEM}: the price to pay is that one has to replace
the Hopf algebra $\Cal H$ by a non-connected bialgebra $\widetilde{\Cal
  H}=S(\Cal T)$ with a suitable coproduct,
such that $\Cal H$ is obtained as the quotient $\widetilde{\Cal H}/\Cal J$,
where $\Cal J$ is the ideal generated by $\bullet-\un$. The substitution product $\star$
then coincides with the one considered in \cite{CHV} via natural
identifications.
%%%%%%
\section{Other related algebraic structures}
%%%%%
\subsection{NAP algebras}
NAP algebras (NAP for Non-Associative Permutative) appear under this name in
\cite{L}, and under the name ``left- (right-)commutative algebras'' in
\cite{DL}. They can be seen in some sense as a ``simplified version'' of
pre-Lie algebras.%%
\subsubsection{Definition and general properties}
A \textsl{left NAP algebra} over a field $k$ is a $k$-vector space $A$ with a
bilinear binary composition $\RHD$ that satisfies the left NAP identity:
\begin{equation}
a\RHD(b\RHD c)=b\RHD(a\RHD c).
\end{equation} 
for any $a,b,c\in A$. Analogously, a \textsl{right NAP algebra} is a
$k$-vector space $A$ with a binary composition $\LHD$ satisfying the right NAP
identity:
\begin{equation}
(a\LHD b)\LHD c=(a\LHD c)\LHD b. 
\end{equation}
 As any right NAP algebra is also a left NAP algebra with product $a\RHD
 b:=b\LHD a$, one can stick to left NAP algebras, what we shall do unless
 specifically indicated.
\subsubsection{Free NAP algebras}
The \textsl{left Butcher product} $s\circ t$ of two rooted trees $s$ and $t$
is defined by grafting $s$ on the root of $t$. For example:
\begin{equation}
\racine\circ\arbrea=\arbrebb,\hskip 12mm \arbrea\circ\arbrea=\arbrecc.
\end{equation}

The following theorem is due to A. Dzhumadil'daev and C. L\"ofwall \cite{DL},
see also \cite{L} for an operadic approach:
\begin{thm}
The free NAP algebra with $d$ generators is the vector space spanned by rooted
trees with $d$ colours, endowed with the left Butcher product.
\end{thm}
\begin{proof}
We give the proof for one generator, the case of $d$ generators being entirely
similar. The left NAP property for the Butcher product is obvious. Let
$(A,\RHD)$ be any left NAP algebra, and let $a\in A$. We have to prove that
there exists a unique left NAP algebra morphism $G_a$ from $(\Cal T,\circ)$ to
$(A,\RHD)$ such that $G_a(\racine)=a$. As in the pre-Lie case, we proceed by
double induction, first on the number $n$ of vertices, second on the number
$k$ of branches. In the case $k=1$ the tree $t$ writes
$B_+(t_1)=t_1\circ\racine$, hence $G_a(t)=G_a(s)\RHD a$ is the only possible
choice. Now a tree with $k$ branches writes:
\begin{equation}
t=B_+(t_1,\ldots,t_k)=t_1\circ\big(t_2\circ\cdots\circ(t_k\circ\racine)...\big).
\end{equation}
The only possible choice is then:
\begin{equation}
G_a(t)=G_a(t_1)\RHD\Big(G_a(t_2)\RHD\cdots\RHD\big(G_a(t_k)\RHD a\big)...\Big),
\end{equation}
and the result is clearly symmetric in $t_1$ and $t_2$ due to the left NAP
identity in $A$. Using the induction hypothesis the result is also invariant
under permutation of the branches $2,3,\ldots,k$. Hence it is invariant under
the permutation of all branches, which proves the theorem.
\end{proof}
Despite the similarity with the pre-Lie situation described in Paragraph
\ref{sect:fpl}, the NAP framework is much simpler due to the set-theoretic
nature of the Butcher product: for any trees $s$ and $t$, the Butcher product $s\circ t$ is a tree
whereas the grafting $s\to t$ is a sum of trees. We obtain for the first
trees:
\begin{eqnarray*}
G_a(\racine)&=&a\\
G_a(\arbrea)&=&a\RHD a\\
G_a(\arbreba)&=&(a\RHD a)\RHD a\\
G_a(\arbrebb)&=&a\RHD(a\RHD a).
\end{eqnarray*}
\subsubsection{NAP algebras of vector fields}
We consider the flat affine $n$-dimensional space $E_n$ although it is possible, through parallel
transport, to consider any smooth manifold endowed with a flat torsion-free
connection. Fix an origin in $E_n$, which will be denoted by $O$. For vector fields $X=\sum_{i=1}^nf_i\partial_i$ and
$Y=\sum_{j=1}^ng_j\partial_j$ we set:
\begin{equation}
X\RHD_{O}Y=\sum_{j=1}^n(X_O.g_j)\partial_j,
\end{equation}
where $X_O:=\sum_{i=0}^n f_i(O)\partial_i$ is the constant vector field
obtained by freezing the coefficients of $X$ at $x=O$.
\begin{prop}
The space $\chi(\R^n)$ of vector fields endowed with product $\RHD_{O}$ is a
left NAP algebra. Moreover, for any other choice of origin $O'\in E_n$, the
conjugation with the translation of vector $\overrightarrow{OO'}$ is an isomorphism from
$\big(\chi(\R^n),\RHD_O)\big)$ onto $\big(\chi(\R^n),\RHD_{O'})\big)$.
\end{prop}
\begin{proof}
Let $X=\sum_{i=1}^nf_i\partial_i$, $Y=\sum_{j=1}^ng_j\partial_j$ and
$Z=\sum_{k=1}^nh_k\partial_k$ be three vector fields. Then:
\begin{equation}
X\RHD_O(Y\RHD_O Z)=\sum_{k=1}^nX_O.(Y_O.h_k)\partial_k
\end{equation}
is symmetric in $X$ and $Y$, due to the fact that the two constant vector
fields $X_O$ and $Y_O$ commute. The second assertion is left as an exercice
for the reader.
\end{proof}
With the notations of Paragraph \ref{sect:vf} there is a unique NAP algebra
morphism
\begin{equation}
\Cal Y_O:(\T,\circ)\longrightarrow\big(\chi(\R^n),\RHD_{O}\big),
\end{equation}
the
\textsl{frozen Cayley map}, such that $\Cal Y_O(\bullet_X)=X$. Considering
also the unique NAP algebra morphism $G_{X,O}=\Cal Y_O\circ d_X:(\Cal T,\circ)\to
\big(\chi(\R^n),\RHD_{O}\big)$, the maps $G_{X,O}(t):\R^n\to\R^n$ are called the
\textsl{frozen elementary differentials}.
\begin{prop}
For any rooted tree $t$, each vertex $v$ being decorated by a vector field
$X_v$, the vector field $\Cal Y_O(t)$ is given at $x\in\R^n$ by the following
recursive procedure: if the decorated tree
$t$ is obtained by grafing all its branches $t_k$ on the root $r$ decorated by the
vector field $X_r=\sum_{i=1}^n f_i\partial_i$, i.e. if it writes
$t=B_+^{X_r}(t_1,\ldots,t_k)$, then:
\begin{align}
\Cal Y_O(\bullet_{X_r})&=X_r,\\
\Cal Y_O(t)&=\sum_{i=1}^n \Cal Y_{O,i}(t)\partial_i \hbox{ with :}\nonumber\\
\Cal Y_{O,i}(t)(x)&=f_i^{(k)}(x)\big(\Cal Y_O(t_1)(O),\ldots,\Cal Y_O(t_k)(O)\big),
\end{align}
where $f_i^{(k)}(x)$ stands for the $k^{\smop{th}}$ differential of $f_i$
evaluated at $x$.
\end{prop}
\begin{proof}
We prove the
result by induction on the number $k$ of branches: for $k=1$ we check:
\begin{eqnarray*}
\Cal Y_O(s\circ\bullet_Y)(x)&=&\big(\Cal Y_O(s)\RHD_O\Cal Y_O(\bullet_Y)\big)(x)\\
&=&\big(\Cal Y_O(s)\RHD_O Y\big)(x)\\
&=&\sum_{j=1}^n\big(\Cal Y_O(s)_O.g_j\big)(x)\partial_j\\
&=&\sum_{j=1}^ng'_j(x)\big(\Cal Y_O(s)(O)\big)\partial_j.
\end{eqnarray*}
Now we can compute, using the induction hypothesis and the fact that the
vector fields $\Cal Y_O(t_j)(O)$ are constant:
\begin{eqnarray*}
\Cal Y_O\big(B_+^Y(t_1,\ldots,t_k)\big)(x)&=& \Cal Y_O\big(t_1\circ B_+^Y(t_2,\ldots,t_k)\big)(x)
\\
&=&\Big(\Cal Y_O(t_1)\RHD_O\sum_{j=1}^n f_i^{(k-1)}\big(\Cal Y_O(t_2)(O),\ldots,\Cal
Y_O(t_k)(O)\big)\partial_j\Big)(x)\\
&=&\sum_{j=1}^nf_i^{(k)}(x)\big(\Cal Y_O(t_1)(O),\ldots,\Cal
Y_O(t_k)(O)\big)\partial_j.
\end{eqnarray*}
\end{proof}
\begin{cor}[closed formula]\label{cf-nap}
With the notations of Corollary \ref{cf}, for any rooted tree $t$ with set of vertices $\Cal V(t)$ and root $r$, each vertex $v$ being decorated by a vector field
$X_v=\sum_{i=1}^nX_v^i\partial_i$, the vector field $\Cal Y_O(t)$ is given at $x\in\R^n$ by the following formula:
\begin{equation}
\Cal Y_O(t)(x)=\sum_{F:\Cal V(t)\to\{1,\ldots ,n\}}\left(\prod_{v\in\Cal V(t)-\{r\}}
\partial_{I(v)}(X_v^{F(v)})(O)\right)\partial_{I(r)}(X_r^{F(r)})(x)\partial_{F(r)}.
\end{equation}
\end{cor}
\subsection{Novikov algebras}
A Novikov algebra is a right pre-Lie algebra which is also left NAP, namely a
vector space $A$ together with a bilinear product $*$ such that for any
$a,b,c\in A$ we have:
\begin{eqnarray}
a*(b*c)-(a*b)*c&=&a*(c*b)-(a*c)*b,\\
a*(b*c)&=&b*(a*c).
\end{eqnarray}
Novikov algebras first appeared in hydrodynamical equations
\cite{BN85,O92}. The prototype is a commutative associative algebra together
with a derivation $D$, the Novikov product being given by:
\begin{equation}
a*b:=(Da)b.
\end{equation}
The free Novikov algebra on a set of generators has been given in
\cite[Section 7]{DL} in terms of some classes of rooted trees.
\subsection{Assosymmetric algebras}
An assosymmetric algebra is a vector space endowed with a bilinear operation which is both left and right pre-Lie, which means that the associator $a*(b*c)-(a*b)*c$ is symmetric under the permutation group $S_3$. This notion has been introduced by E. Kleinfeld as early as 1957 \cite{Kl} (see also \cite{HJP}). All associative algebras are obviously assosymmetric, but the converse is not true.
\subsection{Dendriform algebras}
A {\sl dendriform algebra\/}~\cite{Loday01} over the field $k$ is a $k$-vector space $A$ endowed with two bilinear operations, denoted $\prec$ and $\succ$ and called right and left products, respectively, subject to the three axioms below:
 \allowdisplaybreaks{
\begin{eqnarray}
  (a\prec b)\prec c  &=& a\prec(b \prec c + b \succ c)        \label{A1}\\
  (a\succ b)\prec c  &=& a\succ(b\prec c) 				 \label{A2}\\
   a\succ(b\succ c)  &=& (a \prec b + a \succ b)\succ c        \label{A3}.
\end{eqnarray}}
One readily verifies that these relations yield associativity for the product 
\begin{equation}
\label{dendassoc}
	a * b := a \prec b + a \succ b.  
\end{equation}	
However, at the same time the dendriform relations imply that the bilinear product
$\rhd$ defined by:
\begin{equation}
\label{def:prelie}
    a \rhd b:= a\succ b-b\prec a,
\end{equation}
is left pre-Lie. The associative operation $*$ and the pre-Lie operation $\rhd$ define the same Lie bracket, and this is of course still true for the opposite (right) pre-Lie product $\lhd$:
\begin{equation*}
    [\![a,b]\!]:=a*b-b*a=a\rhd b-b\rhd a=a\lhd b-b\lhd a.
\end{equation*}
In the commutative case (commutative dendriform algebras
are also named {\sl Zinbiel algebras\/} \cite{Loday95}, \cite{Loday01}), the
left and right operations are further required to identify, so that $a \succ b
= b \prec a$. In this case both pre-Lie products vanish. A natural example of
a commutative dendriform algebra is given by the shuffle algebra in terms of
half-shuffles \cite{Sch58}. Any associative algebra $A$ equipped with a linear
integral-like map $I: A \to A$ satisfying the integration by parts rule also
gives a dendriform algebra, when $a \prec b:=aI(b)$ and $a \succ
b:=I(a)b$. The left pre-Lie product is then given by $a\rhd b=[I(a),b]$. It is
worth mentioning that Zinbiel algebras are also NAP algebras, as shown by the
computation below (dating back to \cite{Sch58}):
\begin{eqnarray*}
a\succ(b\succ c)&=&(a\succ b+a\prec b)\succ c\\
&=&(a\succ b+b\succ a)\succ c\\
&=&b\succ(a\succ c).
\end{eqnarray*}
There also exists a twisted versions of dendriform algebras, encompassing
operators like Jackson integral $I_q$ \cite{EM09}. Returning to ordinary dendriform algebras, observe that:
\begin{equation}
	a*b + b \rhd a = a \succ b + b \succ a.
\end{equation}
This identity generalizes to any number of elements, expressing the
symmetrization of $\Big(...\big((a_1\succ a_2)\succ a_3\big)\cdots\Big)\succ
a_n$ in terms of the associative product and the left pre-Lie product
\cite{EMP07}. For more on dendriform algebras and the associated pre-Lie
structures, see \cite{EM, EMdendeq,EM09,EMP07} and K. Ebrahimi-Fard's note in the present volume.
%%%%%%%%%%%%%%%%%%%%%%%%%%%%%%%%%%%%%%%%%%%%%%%%%%%%%%%%%%%%%%%%%%%

\end{document}